\documentclass{amsart}

\usepackage[utf8]{inputenc}
\usepackage[english]{babel}
\usepackage{lmodern}
\usepackage[a4paper, margin=3cm]{geometry}
\usepackage{mathtools}
\usepackage{amsfonts}
\usepackage{amsthm}
\usepackage{amssymb}
\usepackage{listings}
\usepackage{graphicx}
\usepackage{tikz}
\usetikzlibrary{cd}
\usetikzlibrary{babel}
\usepackage{bbold}
\usepackage[colorinlistoftodos,textwidth=2cm]{todonotes}
\usepackage[font=small,labelfont=bf]{caption}
\usepackage[subrefformat=parens]{subcaption}
\usepackage{overpic}
\usepackage{enumerate}
\usepackage[T1]{fontenc}

\usepackage{latexsym}
\usepackage{pgf}
\usepackage{calrsfs,bbm}

\usepackage{amsmath}
\usepackage{theoremref}
\usepackage{mathrsfs}
\usepackage{amsthm}
\usepackage{mathtools}
\usepackage{enumitem}
\usepackage{tikz}
\usetikzlibrary{arrows,decorations.markings,knots,
	hobby,
	decorations.pathreplacing,
	shapes.geometric,
	calc}
\usepackage{float}
\usepackage{pgfkeys}

\usepackage{blkarray}
\usepackage{overpic}
\usepackage{soul}

\usepackage{enumitem}
\setlist{leftmargin=6.5mm}
\setenumerate[1]{label={(\roman*)}}
\setlist[enumerate]{leftmargin=8mm}

\graphicspath{ {images/} }
\usepackage{xcolor}
\usepackage[colorlinks,
    linkcolor={red!50!black},
    citecolor={blue!50!black},
    urlcolor={blue!80!black}]{hyperref}

\newcommand{\Z}{\mathbb{Z}}
\newcommand{\Q}{\mathbb{Q}}

\renewcommand{\Re}{\operatorname{Re}}
\renewcommand{\Im}{\operatorname{Im}}
\DeclareMathOperator{\lk}{lk}
\newcommand{\sgn}{\operatorname{sgn}}
\newcommand{\sign}{\operatorname{sign}}
\newcommand{\nul}{\operatorname{null}}

\newcommand{\om}{\omega}
\newcommand{\bom}{\overline{\omega}}
\newcommand{\som}{\omega^{1/2}}
\newcommand{\bsom}{\omega^{-1/2}}
\newcommand{\ov}{\overline}

\newtheorem{thm}{Theorem}
\newtheorem{lem}{Lemma}
\newtheorem{prop}{Proposition}
\newtheorem{cor}{Corollary}

\theoremstyle{definition}
\newtheorem{definition}{Definition}
\newtheorem{ex}{Example}

\theoremstyle{remark}
\newtheorem*{rems}{Remarks}
\newtheorem{note}{Note}

\makeatletter
\@namedef{subjclassname@1991}{2020 Mathematics Subject Classification}
\makeatother

 \title{A diagrammatic computation of abelian link invariants}

\author{David Cimasoni}
\address{David Cimasoni -- Section de math\'ematiques, Universit\'e de Gen\`eve, Suisse}
\email{david.cimasoni@unige.ch}

\author{Livio Ferretti}
\address{Livio Ferretti -- Section de math\'ematiques, Universit\'e de Gen\`eve, Suisse}
\email{livio.ferretti@unige.ch}

\author{Jessica Liu}
\address{Jessica Liu -- Department of Mathematics, University of Toronto, Canada}
\email{jessliu@math.toronto.edu}

\begin{document}

\makeatletter
   \providecommand\@dotsep{5}
 \makeatother

\begin{abstract}
We show how the multivariable signature and Alexander polynomial of a colored link can be computed from a
single symmetric matrix naturally defined from a colored link diagram. In the case of a single variable, it coincides with
the matrix introduced by Kashaev~\cite{Kas21}, which was recently proven to compute the Levine-Tristram signature
and the Alexander polynomial of oriented links~\cite{Liu23,CimFer23}. As a corollary, we obtain a multivariable extension
of Kauffman's determinantal model of the Alexander polynomial~\cite{Kau83}, recovering a result of Zibrowius~\cite{Zib17}.
\end{abstract}



\maketitle
	

	\section{Introduction}
	\label{sec:intro}
	
As its title suggests, the aim of the present article is to give a way of computing several classical link invariants directly from a diagram. Before specifying these invariants, let us mention that this story is best told in the context of {\em colored links\/}, that we now recall.

Given an integer~$\mu>0$, a{\em~$\mu$-colored link\/} is an oriented link~$L\subset S^3$ each of whose components is
endowed with a {\em color\/} in~$\{1,\dots,\mu\}$ in such a way that all these colors are used.
Two colored links are isotopic if they are related by an ambient isotopy which respects the orientation
and color of all components. Clearly, a~$1$-colored link is just an oriented link, while
a~$\mu$-component~$\mu$-colored link is an oriented ordered link.
A~$\mu$-colored link can be described by an oriented link diagram~$D$ with colored components, an object
which we will refer to as a {\em $\mu$-colored diagram\/} (see Figure~\ref{fig:example} for an example).
As usual, a crossing~$v$ of~$D$ is naturally endowed with a sign that we denote by~$\sgn(v)=\pm 1$, see Figure~\ref{fig:example}.
Finally, a crossing will be called {\em monochromatic\/} if the two corresponding strands are of the same color,
and {\em bichromatic\/} otherwise.

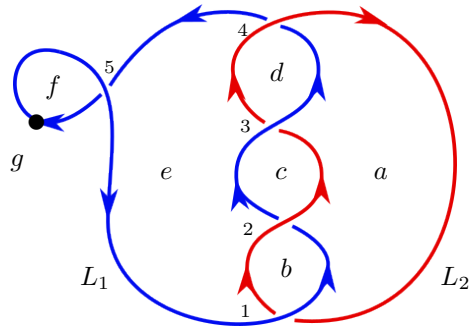
\begin{figure}[tbp]
    \centering
    \begin{picture}(200,130)
    \put(5, 5){\includegraphics[height = 4.25cm]{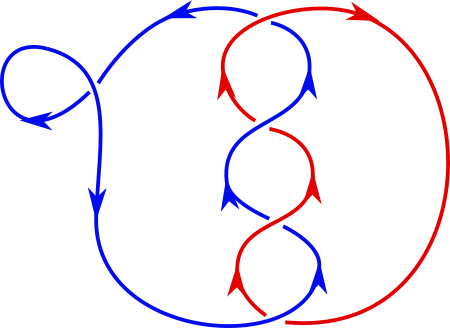}}
    \put(30,20){$L_1$}
    \put(165,20){$L_2$}
    \put(10,78){\scalebox{1.5}{$\bullet$}}
    \put(90,8){$\scriptstyle{1}$}
    \put(91,39){$\scriptstyle2$}
    \put(90,78){$\scriptstyle3$}
    \put(89,114){$\scriptstyle4$}
    \put(39,100){$\scriptstyle5$}
    \put(140,60){$a$}
    \put(105,23){$b$}
    \put(103,60){$c$}
    \put(101,97){$d$}
    \put(60,60){$e$}
    \put(17,92){$f$}
    \put(4,65){$g$}    
    \end{picture}
    \caption{A $2$-colored diagram~$D$ for a $2$-colored link~$L=L_1\cup L_2$.
    The crossings are labelled~1 through~5 and the regions are labelled~$a$ through~$g$.
    Crossings~$1$ to~$4$ are positive and bichromatic, while crossing~$5$ is negative and monochromatic.
    The marked point on~$L_1$ will serve a further purpose.}
    \label{fig:example}
\end{figure}

The invariants we are interested in computing are the classical abelian invariants of a~$\mu$-colored link~$L$, namely
its multivariable {\em signature\/} and {\em nullity\/}~$\sigma_L,\eta_L\colon (S^1\setminus\{1\})^\mu\to\Z$,
and its multivariable {\em Alexander polynomial\/}~$\Delta_L$ in the normalized form given by the {\em Conway function\/}~$\nabla_L$. In the case~$\mu=1$, these are the well-known {\em Levine-Tristram signature\/} and {\em nullity\/} and {\em Alexander-Conway polynomial\/}, without doubt among the most studied of link invariants. We refer to Section~\ref{sec:back}
for the definition of these classical objects.

\medskip

	\begin{figure}[tbp]
    \centering
    \begin{picture}(280,100)
    \put(32,66){$a$}
    \put(4,38){$b$}
    \put(32,16){$c$}
    \put(60,38){$d$}
    \put(-5,0){$j$}
    \put(72,0){$k$}
   \put(0,6){\includegraphics[width=2.5cm]{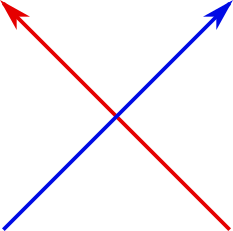}}
   \put(120,30){$\begin{array}{c|cccc}
   &a&b&c&d\\
   \hline
   a&x_{jk}&x_{j}&1&x_k\\
   b&x_j&2x_jx_k-x_{jk}&x_k&1\\
   c&1&x_k&x_{jk}&x_j\\
   d&x_k&1&x_j&2x_jx_k-x_{jk}
   \end{array}$}
   \end{picture}
    \caption{A crossing $v$ together with the corresponding~$4\times 4$ minor of~$\tau_v(x)$. The incoming left strand is of color~$j$, the incoming right strand of color~$k$, and the four adjacent regions are~$a, b, c,$ and~$d$.}
    \label{fig:minor}
\end{figure}
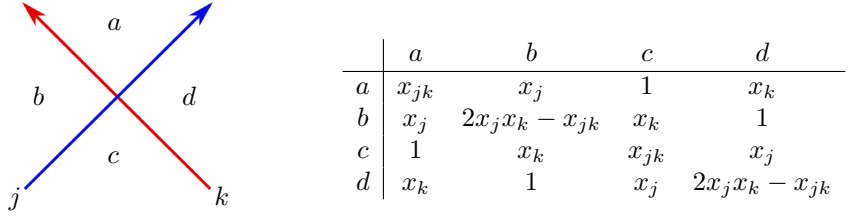

As our main result will show, these invariants can all be computed from a single symmetric matrix~$\tau_D(x)$,
whose coefficients are functions
of formal variables~$x=\{x_j,x_{jk}\mid 1\le j,k\le \mu\}$ indexed by (unordered pairs of) colors.
We now give its definition.

\begin{definition}
\label{def:tau}
Given a~$\mu$-colored diagram~$D$, let~$\tau_D(x)$ be the symmetric matrix with rows and columns indexed by the regions
of~$D$ defined by
\[
\tau_D(x)=\sum_v \frac{\sgn(v)}{\sqrt{1-x_j^2}\sqrt{1-x_k^2}}\tau_v(x)\,,
\]
where the sum is over all crossings of~$D$, the indices~$j,k\in\{1,\dots,\mu\}$ are the (possibly identical) colors of the two strands crossing at~$v$,
and the only non-vanishing coefficients of the matrix~$\tau_v(x)$ are given in Figure~\ref{fig:minor}.

Also, we shall denote by~$\widetilde\tau_D(x)$ the matrix obtained by removing the two rows and columns corresponding to two adjacent regions of~$D$ determined by a marked point on~$D$. We will assume without loss of generality that this point is on
a strand of color~$1$.
\end{definition}

Note that if the regions~$a,b,c,d$ around a crossing~$v$ are not all distinct, then one should add the corresponding rows and columns of~$\tau_v(x)$. This happens in the following example.

\begin{ex}
\label{ex:1}
Consider the~$2$-colored diagram~$D$ illustrated in Figure~\ref{fig:example}. Ordering the regions
alphabetically, the corresponding matrix is given by 
\begin{gather*}
	\tau_D(x)  = \frac{1}{\sqrt{1-x_1^2}\sqrt{1-x_2^2}}\left[\begin{array}{ccccccc} 	4(2x_1x_2-x_{12})& 2x_2 & 2x_1  & 2x_2 & 4 & 0 & 2x_1\\[2pt]
		2x_2 & 2x_{12} & 1 & 0 & 2x_1 & 0 & 1\\[2pt]
		2x_1 & 1 & 2x_{12} & 1 & 2x_2 & 0 & 0\\[2pt]
		2x_2 & 0 & 1 & 2x_{12} & 2x_1 & 0 & 1\\[2pt]
		4 & 2x_1 & 2x_2 & 2x_1 & 4(2x_1x_2-x_{12})& 0 & 2x_2\\[2pt]
		0 & 0 & 0 & 0 & 0 & 0 & 0\\[2pt]
		2x_1 & 1 & 0 & 1 & 2x_2 & 0 & 2x_{12}\end{array} \right]\\
		- \frac{1}{1-x_1^2} \left[\begin{array}{ccccccc} 	0&0&0&0&0&0&0\\[2pt]
		0&0&0&0&0&0&0\\[2pt]
		0&0&0&0&0&0&0\\[2pt]
		0&0&0&0&0&0&0\\[2pt]
		0&0&0&0&2x_1^2-x_{11}&1&2x_1\\[2pt]
		0&0&0&0&1&2x_1^2-x_{11}&2x_1\\[2pt]
		0&0&0&0&2x_1&2x_1&2x_{11}+2\end{array} \right]\,,
\end{gather*}	
where the first summand contains the contributions from the four (positive) bichromatic crossings and the second summand is the contribution from the (negative) monochromatic crossing.
	

\end{ex}

Here is our main result.

\begin{thm}
\label{thm}
Let~$D$ be an arbitrary~$\mu$-colored diagram for a~$\mu$-colored link~$L$.
\begin{enumerate}
\item For any~$\omega=(\omega_1,\dots,\omega_\mu)\in(S^1\setminus\{1\})^\mu$,  the signature and nullity of~$L$ are given by
\[
\sigma_L(\omega)=\textstyle{\frac{1}{2}}(\sign(\widetilde\tau_D(\omega))-w_{\mathrm{m}}(D))\quad\text{and}\quad\eta_L(\omega)=\textstyle{\frac{1}{2}}\nul(\widetilde\tau_D(\omega))\,,
\]
where~$w_{\mathrm{m}}(D)$ is the sum of the signs of all monochromatic crossings of~$D$, and~$\tau_D(\omega)$ stands for the evaluation of~$\tau_D(x)$ at
\[
x_j=\Re(\omega_j^{1/2})\,,\quad x_{jk}=\Re(\omega_j^{1/2}\omega_k^{1/2})\,.
\]
\item If~$D$ is connected, then the Conway function of~$L$ satisfies
\[
\nabla^2_L(t_1,\dots,t_\mu)=\frac{1}{(t_1-t_1^{-1})^2}\bigg(\prod_v-\sgn(v)\frac{t_j-t_j^{-1}}{2}\frac{t_k-t_k^{-1}}{2}\bigg)\cdot\det(\widetilde{\tau}_D(t^2))\,,
\]
where the product is over all crossings of~$D$, the indices~$j,k$ are the (possibly identical) colors of the two strands crossing at~$v$, and~$\tau_D(t^2)$ stands for the evaluation of~$\tau_D(x)$ at
\[
x_j=\frac{t_j+t_j^{-1}}{2}\,,\quad x_{jk}=\frac{t_jt_k+t_j^{-1}t_k^{-1}}{2}\,.
\]
\end{enumerate}
\end{thm}

Several remarks are in order.

\begin{rems}
\label{rem:1}
\begin{enumerate}
\item We need to fix one square root of each coordinate~$\omega_j\in S^1\setminus\{1\}$ of~$\omega$: our choice is to take~$\omega_j = e^{i\theta_j}$ with $\theta_j\in (0,2\pi)$, and $\omega_j^{1/2} = e^{i\theta_j/2}$. 
In other words,~$\omega_j^{1/2}$ denotes the unique square root such that~$\Im(\omega_j^{1/2})$ lies in~$(0,1]$. In particular, we have $(\bar{\omega}_j)^{1/2} = -(\overline{\omega_j^{1/2}})$, and $\sqrt{1-x_j^2} = \Im(\omega_j^{1/2})$.
Note that~$x_j^2\neq 1$,
so~$\tau_D(\omega)$ is a well-defined symmetric {\em real\/} matrix.

\item When defining~$\tau_D(t^2)$ in the second point of Theorem~\ref{thm}, we need to explain how we evaluate $\sqrt{1-x_j^2}$, since there is a sign ambiguity: we set $\sqrt{1-x_j^2} = \frac{t_j^{-1}-t_j}{2i}$. (See also Note~\ref{note:change_variables}.)

\item In both points of Theorem~\ref{thm}, the evaluations of the formal variables satisfy~$x_{jj}=2x_j^2-1$
for all~$j$. Therefore, if a crossing~$v$ is monochromatic, then the matrix~$\tau_v(x)$ can be written in a simple form
which only depends on the single variable~$x_j$.

\item In particular, if~$\mu=1$, then~$\tau_D(x)$ depends on a single variable. This matrix was first introduced by Kashaev in~\cite{Kas21}, see discussion below.

\item In principle, it would be enough to only define the matrix~$\tau_D(t^2)$ to state and prove Theorem~\ref{thm}. The reason we chose to introduce the matrix $\tau_D(x)$ is merely historical: in doing so, we explicitly present our results as an extension of Kashaev's work.

\item In~\cite{Kas21}, Kashaev studied the effect of Reidemeister moves on the matrix~$\tau_D(x)$ in the single-variable case, proving that a certain \textit{modified S-equivalence} class of the matrix is a link invariant.
We expect a similar result to hold in the multivariable setting. 

\item 
As it will become apparent from the proof of the second point, our construction naturally produces the square of the Conway function, so we cannot hope to recover its sign. (See also Note~\ref{note:sign}.)
\end{enumerate}
\end{rems}

\begin{ex}
\label{ex:2}
Consider once again the~$2$-colored link illustrated in Figure~\ref{fig:example}. From the corresponding matrix~$\tau_D(x)$ given in Example~\ref{ex:1}, we compute 
\begin{equation*}
	\footnotesize
	\setlength{\arraycolsep}{-1pt}
	\medmuskip = 2mu
	\widetilde{\tau}_D(t^2) = \frac{-4}{(t_1-t_1^{-1})(t_2-t_2^{-1})}\left[\begin{array}{ccccc} 	2(t_1t_2^{-1}+t_1^{-1}t_2)& t_2+t_2^{-1} & t_1+t_1^{-1} & t_2+t_2^{-1} & 4 \\[2pt]
		t_2+t_2^{-1} & t_1t_2+t_1^{-1}t_2^{-1} & 1 & 0 & t_1+t_1^{-1} \\[2pt]
		t_1+t_1^{-1} & 1 & t_1t_2+t_1^{-1}t_2^{-1} & 1 & t_2+t_2^{-1} \\[2pt]
		t_2+t_2^{-1} & 0 & 1 & t_1t_2+t_1^{-1}t_2^{-1} & t_1+t_1^{-1} \\[2pt]
		4 & t_1+t_1^{-1} & t_2+t_2^{-1} & t_1+t_1^{-1} & 2(t_1t_2^{-1}+t_1^{-1}t_2)-\frac{t_2-t_2^{-1}}{t_1-t_1^{-1}}\end{array} \right]
\end{equation*}
and obtain $$\det(\widetilde{\tau}_D(t^2)) = -\left( \frac{-4}{(t_1-t_1^{-1})(t_2-t_2^{-1})} \right)^5(t_1-t^{-1}_1)(t_2-t_2^{-1})(t_1t_2+t_1^{-1}t_2^{-1})^2\,.$$
In order to compute $\sign(\widetilde\tau_D(\omega))$, we evaluate~$\widetilde{\tau}_D(t^2)$ at~$t_i =\som_i$, and recall that the signature can change value only when the nullity changes value. Since~$\Im(\omega_j^{1/2}) \in (0,1]$, the computation of the determinant immediately implies that the nullity of~$\widetilde\tau_D(\omega)$
vanishes on the complement of $\Sigma:=\{(\omega_1,\omega_2)\in(S^1\setminus\{1\})^2\mid \omega_1\omega_2=-1\}$, so its signature is constant on the connected components of this complement.
Then, an easy but tedious computation of minors shows that the nullity of~$\widetilde\tau_D(\omega)$ is constant equal to~2 on~$\Sigma$, so the signature is also constant along each component of~$\Sigma$. In particular, Theorem~\ref{thm} implies that $\eta_L(\omega) = 1$ if $\omega_1\omega_2 = -1$, and $\eta_L(\omega) = 0$ otherwise. As for the signature, the values of $\sign(\widetilde\tau_D(\omega))$ can now be computed by picking one point in each of those connected components. For example, taking $\omega_1 = \omega_2 = -1$ we obtain 
\begin{equation*}
	\widetilde{\tau}_D(\omega) = \left[\begin{array}{ccccc}
	4 & 0 & 0 & 0 & 4\\[2pt]
	0 & -2 & 1 & 0 & 0 \\[2pt]
	0 & 1 & -2 & 1 & 0\\[2pt]
	0 & 0 & 1 & -2 & 0\\[2pt]
	4 & 0 & 0 & 0 & 3\\[2pt]	\end{array} \right]
\end{equation*}
from which we compute $\sign(\widetilde\tau_D(\omega)) = -3$ and hence, by Theorem~\ref{thm}, $\sigma_L(\omega) = -1$. The other values of $\sigma_L$ can be computed in a similar way. The result can be summarized in the formula $\sigma_L(\omega_1,\omega_2)=-\sign\left[\Re((1-\omega_1)(1-\omega_2))\right]$, and is plotted below, where the domain $(S^1\setminus\{1\})^2$ is represented as a square.
\begin{figure}[h]
    \centering
    \begin{picture}(80,80)
    \put(0,0){\includegraphics[width=2.5cm]{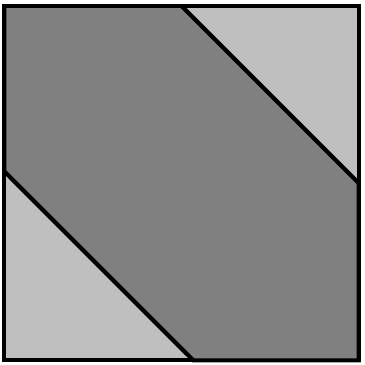}}
    \put(10,10){1}
    \put(58,58){1}
    \put(33,33){-1}
    \put(-4,18){\scalebox{1.5}{$\nearrow$}}
    \put(62,45.5){\scalebox{1.5}{$\swarrow$}}
    \put(-10, 10){0}
    \put(78,55){0}
       \end{picture}
\end{figure}

\noindent Finally, we also get~$\pm\nabla_L(t_1,t_2)=t_1t_2+t_1^{-1}t_2^{-1}$.
\end{ex}

The first point of this theorem provides a practical new way of computing multivariable signatures, but it also yields
much simpler proofs of known properties of this invariant. For example, consider the following situation: let~$L$ be the~$(\mu-1)$-colored link obtained from a~$\mu$-colored link~$L'$ by identifying the colors~$\mu-1$ and~$\mu$; then, for all~$(\omega_1,\dots,\omega_{\mu-1})\in(S^1\setminus\{1\})^{\mu-1}$, we have
the equality
\[
\sigma_{L}(\omega_1,\dots,\omega_{\mu-1})=\sigma_{L'}(\omega_1,\dots,\omega_{\mu-1},\omega_{\mu-1})-\sum\lk(K_{\mu-1},K_{\mu})\,,
\]
the sum being over all components~$K_{\mu-1}\subset L$ of color~$\mu-1$ and all components~$K_{\mu}\subset L$ of color~$\mu$. The original proof of this fact is rather tedious, see~\cite[Proposition~2.5]{CimFlo08}. It is an amusing exercise
to check that this fact immediately follows from Theorem~\ref{thm}.
In particular, given a~$\mu$-component~$\mu$-colored link~$L'=K_1\cup\dots\cup K_\mu$, 
the underlying~$1$-colored link~$L$ satisfies the equality
\[
\xi(L):=\sigma_L(-1)+\sum_{i<j}\lk(K_i,K_j)=\sigma_{L'}(-1,\dots,-1)\,.
\]
This, together with the straightfoward~\cite[Proposition~2.8]{CimFlo08}, yields a one line proof of the main result of~\cite{Mur70}: the integer~$\xi(L)$ does not depend on the orientation of the components of~$L$.

\medskip

The second point of this theorem implies a corollary that we now present.
Given a connected colored diagram~$D$, let~$K_D$ be the matrix whose rows are indexed by the crossings of~$D$,
whose columns are indexed by the regions of~$D$, and whose coefficients are defined by the label of the corners in Figure~\ref{fig:labels}. (If a region abuts a corner from two sides, then the corresponding labels should be added.) Finally, let~$\widetilde K_D$ denote the square matrix obtained from~$K_D$ by removing two columns corresponding to two adjacent regions (separated by a strand of color~$1$).

	\begin{figure}[tbp]
    \centering
   \begin{picture}(100,100)
    \put(24,80){\small$\left(t_j^{1/2}t_k^{1/2}\right)^s$}
    \put(-17,45){\small$\left(t_j^{1/2}t_k^{-1/2}\right)^s$}
    \put(17,8){\small$\left(t_j^{-1/2}t_k^{-1/2}\right)^s$}
    \put(60,45){\small$\left(t_j^{-1/2}t_k^{1/2}\right)^s$}
    \put(-5,0){$j$}
    \put(91,0){$k$}
   \put(0,6){\includegraphics[width=3.2cm]{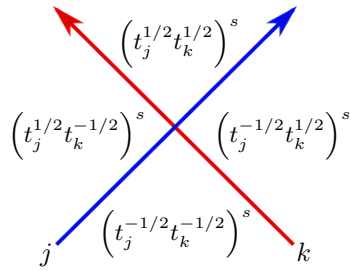}}
   \end{picture}
    \caption{The labels in the definition of~$K_D$ around a vertex~$v$ with~$s=\sgn(v)$.}
    \label{fig:labels}
\end{figure}

\begin{cor}
\label{cor}
If~$D$ is a connected diagram for a colored link~$L$, then
\[
\det\widetilde K_D=\pm (t_1-t_1^{-1})\nabla_L(t_1,\dots,t_\mu)\,.
\]
\end{cor}

\begin{note}
\label{note:sign}
	In fact, what we get from Corollary~\ref{cor} is the symmetrized Alexander polynomial: since the sign of the determinant depends on the order of the rows and columns of $\widetilde K_D$, i.e. on the numbering of the regions and crossings of $D$, we cannot determine the sign of the Conway function.
\end{note}

\begin{ex}
\label{ex:3}
Consider one last time the~$2$-colored link illustrated in Figure~\ref{fig:example}. The matrix~$K_D$ is given by
\begin{equation*}
	\left[\begin{array}{ccccccc}
		t_1^{-1/2}t_2^{1/2} & t_1^{1/2}t_2^{1/2} & 0 & 0 & t_1^{1/2}t_2^{-1/2} & 0 & t_1^{-1/2}t_2^{-1/2}\\[3pt]
		t_1^{1/2}t_2^{-1/2} & t_1^{-1/2}t_2^{-1/2} & t_1^{1/2}t_2^{1/2} & 0 & t_1^{-1/2}t_2^{1/2} & 0 &0\\[3pt]
		t_1^{-1/2}t_2^{1/2} & 0 & t_1^{-1/2}t_2^{-1/2} & t_1^{1/2}t_2^{1/2} & t_1^{1/2}t_2^{- 1/2} & 0 & 0\\[3pt]
		t_1^{1/2}t_2^{-1/2} & 0 & 0 & t_1^{-1/2}t_2^{-1/2} & t_1^{-1/2}t_2^{1/2} & 0 & t_1^{1/2}t_2^{1/2}\\[3pt]
		0 & 0 & 0 & 0 & 1 & 1 & t_1+t_1^{-1}\end{array} \right]\,,
\end{equation*}
from which we compute once again~$\pm\nabla_L(t_1,t_2)=t_1t_2+t_1^{-1}t_2^{-1}$.
\end{ex}

Let us now put our results in the context of the preexisting literature.

In 2018, Kashaev~\cite{Kas21} defined the matrix~$\tau_D(x)$ in the case~$\mu=1$, and conjectured Theorem~\ref{thm} in this special case. Recently, the first two named-authors~\cite{CimFer23} provided a proof of the second part of
this conjecture by establishing
a connection with Kauffman's determinantal model of the Alexander polynomial~\cite{Kau83}; they also
proved the first part of the Kashaev conjecture in a very restrictive case and indirect way.  Immediately afterwards, the third named-author~\cite{Liu23} gave a complete proof of the
conjecture. Joining our efforts, we now extend the approach of~\cite{Liu23} to the general multivariable case in Theorem~\ref{thm} and in its proof.

As for Corollary~\ref{cor}, in the case~$\mu=1$ it is nothing but Kauffman's aforementioned model for the Alexander polynomial~\cite{Kau83}. Interestingly, Kauffman did state a multivariable version of his model (a detailed proof was only given many years later by Sato~\cite{Sat11}), but it is different from our model. However, Zibrowius~\cite{Zib19} gave a state sum model for the multivariable Conway function which uses the same labels as the ones of Figure~\ref{fig:labels} up to a sign; using an extension of Kauffman's {\em Clock Theorem\/}~\cite{Kau83}, this state sum model can be
turned into a determinantal model which coincides with the one of Corollary~\ref{cor}.
This latter fact can be found in Zibrowius's PhD thesis~\cite[Chapter I.4]{Zib17} (but not in \cite{Zib19}). Therefore, and even though our proof is completely different, Corollary~\ref{cor} is not a new result in the strict sense.

Let us finally mention that Friedl-Kausik-Quintanilha~\cite{Fri22} recently provided an algorithm for the computation of generalized Seifert matrices (see Section~\ref{sub:Seifert}) for colored links given as closures of colored braids. Since such matrices can be used to define~$\sigma_L,\eta_L$ and~$\nabla_L$, this method yields an algorithmic computation of these invariants. However,
the remarkable feature of Theorem~\ref{thm} remains: a new way of computing these invariants from a single symmetric matrix obtained directly from a colored diagram.

\medskip

This paper is organised as follows. In Section~\ref{sec:back}, we recall the necessary background
on generalized Seifert matrices (Section~\ref{sub:Seifert}),  multivariate signatures of colored links (Section~\ref{sub:sign}), and the Conway function (Section~\ref{sub:Conway}). Section~\ref{sec:proof} contains the proof of our results, namely the first and second points of Theorem~\ref{thm} in Section~\ref{sub:proof-sign} and~\ref{sub:proof-Conway}, respectively, and of Corollary~\ref{cor} in Section~\ref{sub:Kauffman}.
A slightly informal last Section~\ref{sub:module} contains results on the Alexander module.

\subsection*{Acknowledgments}
We thank Claudius Zibrowius for informing us of his work on the Kauffman model,
and the anonymous referee for very valuable suggestions.
DC and LF are supported by the Swiss NSF grant 200021-212085. JL is partially supported by NSERC CGS-D.

	\section{Background}
	\label{sec:back}
	
The aim of this section is to briefly recall the necessary background for our work: we start in Section~\ref{sub:Seifert} with the definition of C-complexes and generalized Seifert forms, then move on to multivariate signatures in Section~\ref{sub:sign},  before dealing with the Conway function in Section~\ref{sub:Conway}.

	\subsection{Generalized Seifert surfaces and matrices}
	\label{sub:Seifert}
	
Seifert surfaces and matrices are well known tools in the construction and study of (single-variable) abelian link invariants,
such as the Levine-Tristram signature and the Alexander polynomial. Less well-known is the fact that multivariate invariants can be defined and studied via generalized Seifert surfaces, known as C-complexes. We now introduce these objects, following~\cite{Coo82, Cim04}.
	
To do so, we will use the notation~$L=L_1\cup\dots\cup L_\mu$ for a~$\mu$-colored link,
where~$L_i$ is the sublink of~$L$ consisting of all the components of color~$i$.
	
\begin{figure}[tbp]
    \centering
   \includegraphics[width=8cm,angle=180]{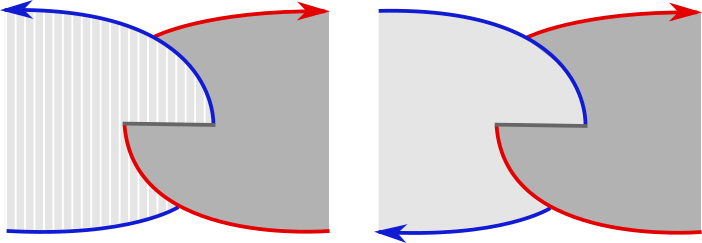}

    \caption{A positive clasp intersection (left), and a negative one (right).}
    \label{fig:clasp}
\end{figure}

\begin{definition}
\label{def:C-cplx}
A {\em C-complex\/} for a~$\mu$-colored link~$L=L_1\cup\dots\cup L_\mu$ is a union~$S=S_1\cup\dots\cup S_\mu$ of surfaces embedded in~$S^3$ satisfying the following conditions:
\begin{enumerate}
    \item for all~$i$, the surface~$S_i$ is a (possibly disconnected) Seifert surface for~$L_i$;
    \item for all~$i\neq j$, the surfaces~$S_i$ and~$S_j$ are either disjoint or intersect in a finite number of {\em clasps\/}, see Figure~\ref{fig:clasp};
    \item for all~$i,j,k$ pairwise distinct, the intersection~$S_i\cap S_j\cap S_k$ is empty.
\end{enumerate}
\end{definition}

Note that a C-complex for a~$1$-colored link~$L$ is nothing but a (possibly disconnected) Seifert surface for the oriented link~$L$.
The existence of a C-complex for any given colored link is easy to establish~\cite{Cim04}. On the other hand, the corresponding notion of S-equivalence is more difficult to prove, see~\cite{DMO} for the recently corrected statement.

These C-complexes allow us to define {\em generalized Seifert forms\/} as follows. For any choice of signs~$\varepsilon=(\varepsilon_1,\dots,\varepsilon_\mu)\in\{\pm 1\}^\mu$, let
\[
\alpha^\varepsilon\colon H_1(S)\times H_1(S)\longrightarrow\mathbb{Z}
\]
be the bilinear form given by~$\alpha^\varepsilon(x,y)=\lk(x^\varepsilon,y)$,
where~$x^\varepsilon$ denotes a well-chosen representative of the homology class~$x\in H_1(S)$ pushed-off~$S_i$ in the~$\varepsilon_i$-normal direction (see~\cite{CimFlo08} for a more formal definition). We denote by~$A^\varepsilon$ the corresponding {\em generalized Seifert matrices\/}, defined with respect to a fixed basis of~$H_1(S)$.
One easily checks the equality
\begin{equation}
\label{eq:A-sym}
A^{-\varepsilon}=(A^\varepsilon)^{\mathrm{T}}
\end{equation}
for all~$\varepsilon\in\{\pm 1\}^\mu$.
Note that the two generalized Seifert matrices~$A^-,A^+$ of a~$1$-colored link~$L$ coincide with the usual Seifert matrix~$A$ of the oriented link~$L$ and its transposed matrix~$A^{\mathrm{T}}$.

The general principle regarding these matrices is the following: what Seifert matrices can do in one variable for oriented links,
generalized Seifert surfaces can do in~$\mu$-variables for~$\mu$-colored links. In Sections~\ref{sub:sign} and~\ref{sub:Conway} we illustrate this principle with two
examples of invariants.

	\subsection{Signatures and nullities of colored links}
	\label{sub:sign}
	
Fix a C-complex~$S$ for a~$\mu$-colored link~$L$ and a basis of~$H_1(S)$.
Consider an element~$\omega=(\omega_1,\dots,\omega_\mu)$ of~$\mathbb{T}^\mu_*\coloneqq(S^1\setminus\{1\})^\mu$, and set
\begin{equation}
\label{eq:H}
H(\omega)\coloneqq\sum_{\varepsilon\in\{\pm 1\}^\mu}\bigg(\prod_{i=1}^\mu(1-\overline{\omega}_i^{\varepsilon_i})\bigg)A^\varepsilon\,.
\end{equation}
Using~\eqref{eq:A-sym}, one easily checks that~$H(\omega)$ is a Hermitian matrix and hence admits a well-defined signature~$\sign(H(\omega))\in\mathbb{Z}$ and nullity~$\nul(H(\omega))\in\mathbb{Z}_{\ge 0}$.

\begin{definition}[\cite{CimFlo08}]
\label{def:sign}
The {\em signature\/} and {\em nullity\/} of the~$\mu$-colored link~$L$ are functions
\[
\sigma_L,\eta_L\colon \mathbb{T}^\mu_*\longrightarrow \Z
\]
defined by~$\sigma_L(\omega)\coloneqq\sign(H(\omega))$ and~$\eta_L(\omega)\coloneqq\nul(H(\omega))$, respectively.
\end{definition}
The fact that these functions are well-defined invariants, i.e. do not depend on the choice of the C-complex~$S$ for~$L$, relies on the aforementioned notion of S-equivalence~\cite{CimFlo08,DMO}. Note that in the case~$\mu=1$, the functions~$\sigma_L,\eta_L\colon S^1\setminus\{1\}\to\Z$ are the signature and nullity of the Hermitian matrix~$(1-\omega)A+(1-\overline{\omega})A^{\mathrm{T}}$, i.e. they coincide with the
Levine-Tristram signature and nullity of the oriented link~$L$. We refer to the recent survey~\cite{Con19} for
background on this classical invariant.

In a nutshell, all the remarkable properties of the Levine-Tristram signature extend to the multivariable setting.
For example, the function~$\sigma_L$ is constant on the connected components
of the complement in~$\mathbb{T}^\mu_*$ of the zeros of the multivariable
Alexander polynomial~$\Delta_L(t_1,\dots,t_\mu)$~\cite{CimFlo08} (see Section~\ref{sub:Conway} below). Also, if~$(\omega_1,\dots,\omega_\mu)\in \mathbb{T}^\mu_*$ is not the root of any Laurent polynomial~$p(t_1,\dots,t_\mu)$ with~$p(1,\dots,1)=\pm 1$, then~$\sigma_L(\omega_1,\dots,\omega_\mu)$ and~$\eta_L(\omega_1,\dots,\omega_\mu)$ are invariant under {\em topological concordance\/} of colored links~\cite{CNT}.

	\subsection{The Conway function of a colored link}
	\label{sub:Conway}
	
	The one-variable Alexander polynomial~$\Delta_L(t)$ of an oriented link~$L$ can be generalized to
	a~$\mu$-variable polynomial invariant~$\Delta_L(t_1,\dots,t_\mu)$ of a~$\mu$-colored link~$L$, a fact known to
	Alexander himself~\cite{Alex28}. To do so, consider the exterior~$X_L:=S^3\setminus\nu(L)$ of~$L$ and the surjective group homomorphism
	\[
	\pi_1(X_L)\to\Z^\mu\,,\quad[\gamma]\mapsto(\lk(\gamma,L_1),\dots,\lk(\gamma,L_\mu))\,.
	\]
	This defines a regular~$\Z^\mu$-cover~$\widehat{X}_L$ of~$X_L$ whose homology groups are hence equipped
	with the structure of a module over the group ring~$\Z[\Z^\mu]=\Z[t_1^{\pm 1},\dots,t_\mu^{\pm 1}]$. In particular,
	the module $\mathcal{A}_L:=H_1(\widehat{X}_L)$ is called the (multivariable) {\em Alexander module\/} of~$L$ (see Section~\ref{sub:module}), and a greatest common
	divisor of the elements of its first elementary ideal is the {\em Alexander polynomial\/} of~$L$.
	
	Note however that this Laurent polynomial in~$\mu$-variables
	is only well-defined up to multiplication by units of the ring~$\Z[t_1^{\pm 1},\dots,t_\mu^{\pm 1}]$, i.e. up to a sign and
	powers of the variables. This later indeterminacy can be easily overcome by harnessing the symmetry of~$\Delta$ and requiring it to satisfy~$\Delta_L(t_1^{-1},\dots,t_\mu^{-1})=\pm \Delta_L(t_1,\dots,t_\mu)$, but the sign issue is a non-trivial one.
	
	The solution was suggested by Conway in his landmark paper~\cite{Con70}. He claimed the existence of a well-defined rational function~$\nabla_L$ satisfying
	\begin{equation}
	\label{eq:Con-Alex}
	\nabla_L(t_1,\dots,t_\mu)\stackrel{\boldsymbol{\cdot}}{=}\begin{cases}
	\frac{1}{t_1-t_1^{-1}}\Delta_L(t^2_1)&\text{if~$\mu=1$;}\\
	\Delta(t_1^2,\dots,t_\mu^2)&\text{if~$\mu>1$,}
\end{cases}
	\end{equation}
where~$\stackrel{\boldsymbol{\cdot}}{=}$ stands for the equality up to multiplication by~$\pm t_1^{\nu_1}\cdots t_\mu^{\nu_\mu}$ with~$\nu_1,\dots,\nu_\mu\in\Z$.
The first explicit construction of this {\em Conway function} was given by Hartley~\cite{Har83} using free differential calculus, but we will make use of the following geometric construction~\cite{Cim04}.
Given any \emph{connected} C-complex $S=S_1\cup\dots\cup S_\mu$ for a $L$, consider the matrix
\begin{equation}
\label{eq:A}
A_S\coloneqq\sum_{\varepsilon\in\{\pm 1\}^\mu}\bigg(\prod_{i=1}^\mu\varepsilon_it_i^{\varepsilon_i}\bigg)A^\varepsilon\,.
\end{equation}
Then, the Conway function of~$L$ is given by
\begin{equation}
\label{eq:Conway}
	\nabla_L(t_1,\dots,t_\mu)=\operatorname{sgn}(S)\bigg(\prod_{i=1}^\mu (t_i-t_i^{-1})^{\chi(S\setminus S_i)-1}\bigg)\det(-A_S)\,,	\end{equation}
	where~$\operatorname{sgn}(S)$ denotes the product of the signs of the clasps of~$S$ (recall Figure~\ref{fig:clasp}).
	Note that in the case~$\mu=1$, equations~\eqref{eq:Con-Alex} and~\eqref{eq:Conway} lead to the formula~$\Delta_L(t)=\det(t^{-1/2}A-t^{1/2}A^{\mathrm{T}})$, the classical definition of the Alexander-Conway polynomial
	of the oriented link~$L$~\cite{Kau81}.
	
	This geometric construction of the Conway function yields straightforward proofs of the various properties
	of this invariant. In particular, it yields a ``geometric explanation'' of the local relations that can be used to compute it from a link diagram, see~\cite{Cim04} for more details.

	\section{Proofs of the main results}
	\label{sec:proof}

As evident from its title, this section contains the proofs of our results.
More precisely, we start in Section~\ref{sub:proof-sign} with the demonstration of the first part of Theorem~\ref{thm}
on signatures and nullities. Section~\ref{sub:proof-Conway} deals with the second part on the Conway function, while Section~\ref{sub:Kauffman} contains the proof of Corollary~\ref{cor} on the
multivariable Kauffman model. Finally, Section~\ref{sub:module} consists in a slightly informal discussion
on the Alexander module.
	
\subsection{Signatures and nullities}
\label{sub:proof-sign}

In this section we discuss how to compute the multivariable signature, proving part (i) of Theorem~\ref{thm} which we now restate for convenience. 
\begin{prop}
\label{prop:sig}
Let~$D$ be a diagram for a $\mu$-colored link~$L$. 
For any~$\omega=(\omega_1,\dots,\omega_\mu)\in(S^1\setminus\{1\})^\mu$,  the signature and nullity of~$L$ are given by
\[
\sigma_L(\omega)=\textstyle{\frac{1}{2}}(\sign(\widetilde\tau_D(\omega))-w_{\mathrm{m}}(D))\quad\text{and}\quad\eta_L(\omega)=\textstyle{\frac{1}{2}}\nul(\widetilde\tau_D(\omega))\,,
\]
where~$w_{\mathrm{m}}(D)$ is the sum of the signs of all monochromatic crossings of~$D$, and~$\tau_D(\omega)$ stands for the evaluation of~$\tau_D(x)$ at~$x_j=\Re(\omega_j^{1/2})$ and~$x_{jk}=\Re(\omega_j^{1/2}\omega_k^{1/2})$.
\end{prop}
\begin{proof} Fix an arbitrary~$\mu$-colored link~$L$, and let~$rL$ denote the~$\mu$-colored  link~$L$ with reverse orientation but same coloring as~$L$. Let~$L \#_{1} rL$ denote a connected sum of~$L$ and~$rL$ along two components of color~1.
Unlike for knots, the isotopy type of the connected sum of (colored) links is not well-defined.
However, any two such connected sums have the same signature and nullity, as these invariants behave additively
under this ill-defined operation: this follows from Propositions~2.12 and 2.5 of~\cite{CimFlo08}. Since the signature and nullity
are unchanged by reversing the orientation (see~\cite[Corollary~2.9]{CimFlo08}), the relations 
\begin{equation}\label{eq:double}\sigma_{L \#_{1} rL}(\omega) = 2\sigma_L(\omega) \hspace*{1cm} \text{and} \hspace*{1cm} \eta_{L \#_{1} rL}(\omega) = 2\eta_L(
\omega)\end{equation}
hold for any such connected sum.

The idea of the proof is to use a diagram~$D$ for~$L$ to construct a C-complex~$S$ for~$L\#_1 rL$ whose first homology has a basis given by classes of loops corresponding to the regions and the monochromatric crossings of~$D$ -- minus the two regions near the connected sum. By taking generalized Seifert matrices with respect to this basis, we show that the matrix~$H(\omega)$ used in Definition~\ref{def:sign} is congruent to a block-diagonal matrix of the form~$\widetilde{\tau}_D(\omega)\oplus Z$ with~$\sigma(Z)=-w_{\mathrm{m}}(D)$ and~$\nul(Z)=0$. Combined with Equation~\eqref{eq:double}, this completes the proof.

We now give the details. To construct a C-complex for~$L\#_1 rL$ from~$D$, we use the following procedure (see Figure~\ref{fig:Ccrossingconstruction} for the construction near crossings, and Figure~\ref{fig:exampleS} for an example).
\begin{enumerate}
\item At each crossing of~$D$, draw a copy of the corresponding crossing for~$rL$ ``a bit above and behind'' the crossing of~$D$.

\begin{figure}[tbp]
    \centering
   \includegraphics[width=\textwidth]{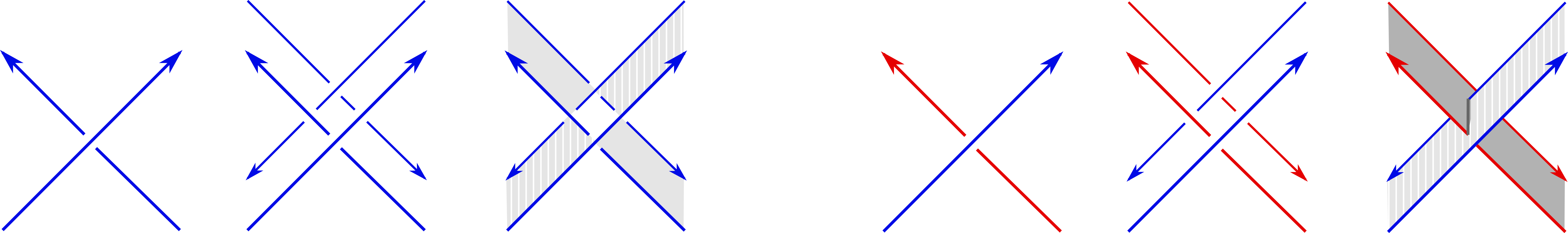}
    \caption{The construction of a C-complex~$S$ for~$L\#_1rL$ near a monochromatic crossing (left) and a bichromatic crossing (right).}
    \label{fig:Ccrossingconstruction}
\end{figure}

\begin{figure}[tbp]
    \centering
   \includegraphics[width=12cm]{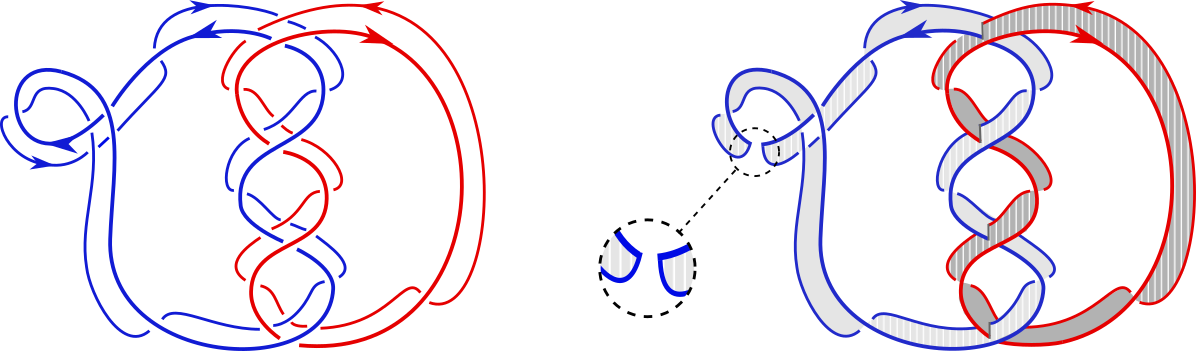}
    \caption{Left: the diagram for the disjoint sum of~$L$ and~$rL$ obtained from~$D$ of Figure~\ref{fig:example}.
    Right: the corresponding C-complex~$S$ for~$L\#_1rL$.}
    \label{fig:exampleS}
\end{figure}

\item Connect the remaining strands of~$rL$ to each other following along the edges of~$D$, possibly creating an additional crossing along each edge (with~$rL$ passing under~$L$). This yields a diagram for the disjoint sum of~$L$ and~$rL$.

\item Create a clasp intersection near each bichromatic crossing of $D$ and apply the usual Seifert algorithm near each monochromatic crossing of $D$.
This yields a C-complex for the disjoint sum of~$L$ and~$rL$.

\item Finally, pick a point on a strand of color~1 in~$D$ and cut the corresponding surface at that place.
The result is a C-complex~$S$ for (some version of)~$L\#_1rL$.

\end{enumerate}

Note that~$S$ deformation retracts onto a graph defined as follows:
take the~4-regular graph underlying the diagram~$D$, add a loop at each vertex corresponding to a monochromatic vertex, and remove the edge along which the connected sum was performed.
As a consequence, a natural basis for~$H_1(S)$ is given by classes of cycles corresponding to the regions and monochromatic crossings of $D$ -- minus the two regions adjacent to where the connected sum happens.
We use the convention that the cycles representing our basis of~$H_1(S)$ are oriented counterclockwise
in the plane of~$D$ where~$S$ is drawn, and denote by the same letter a region or monochromatic crossing and its corresponding cycle of~$H_1(S)$.
Using this explicit basis, we now study the local contribution to generalized Seifert matrices near crossings of~$D$. 

For the remainder of the proof, we adopt the following convention: we say a crossing~$v$ has color~$(j,k)=(j_v, k_v)$
if its incoming left strand is color~$j$, and incoming right strand has color~$k$,
as in Figure~\ref{fig:minor}. If~$j=k$, we may simply say it has color~$j$.

\begin{figure}[tbp]
    \centering
    \begin{picture}(400,100)
    
   \put(0,0){\includegraphics[width=14cm]{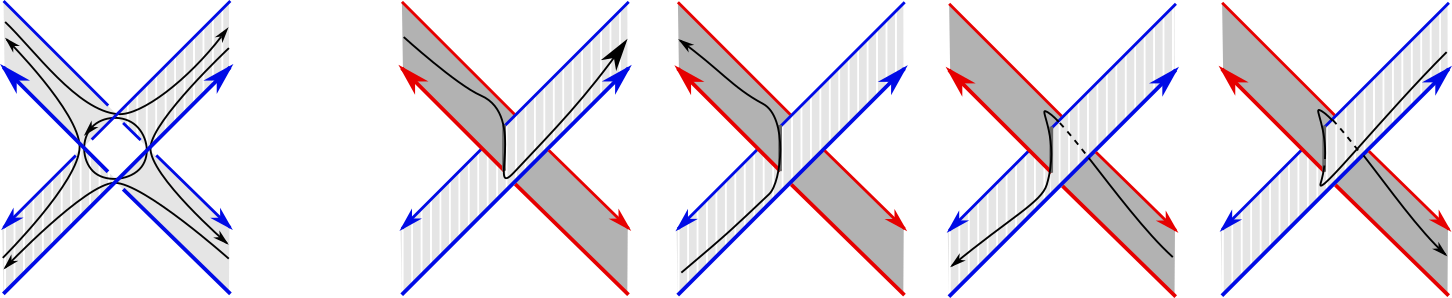}}
    \put(30,70){$a$}
    \put(0,40){$b$}
    \put(30,10){$c$}
    \put(60,40){$d$}
    \put(29,39){$v$}
    \put(138,70){$a$}
    \put(190,40){$b$}
    \put(290,10){$c$}
    \put(390,40){$d$}
    \end{picture}
    \caption{The five cycles of~$S$ near a monochromatic crossing (left, one image) and the four cycles of~$S$ near a bichromatic crossing (right, four images). The labels $a, b, c, d, v$ for the regions are used for the local linking matrices in Figure~\ref{fig:locallinking}.}
    \label{fig:H1classes}
\end{figure}
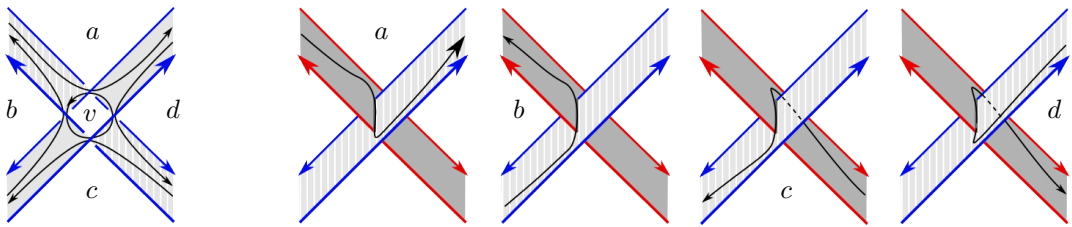

As illustrated in Figure~\ref{fig:H1classes}, there are five homology classes in~$H_1(S)$ coming into play near a monochromatic crossing of~$D$, and four near a bichromatic crossing. To compute the Seifert forms locally, we need to choose a convention for drawing the pushouts so that no contribution to the linking numbers comes from the crossings that occur along the edges of~$D$: this is illustrated in Figure~\ref{fig:edgenolinking}.

\begin{figure}
	\includegraphics[trim = 0 30 0 0, clip, width=1.2cm]{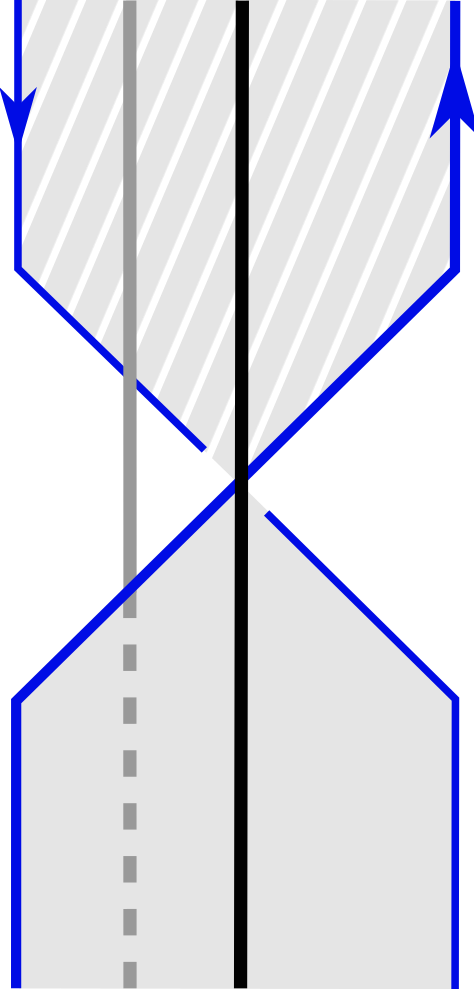}
	\hspace*{1cm}
	\includegraphics[trim = 0 30 0 0, clip, width=1.2cm]{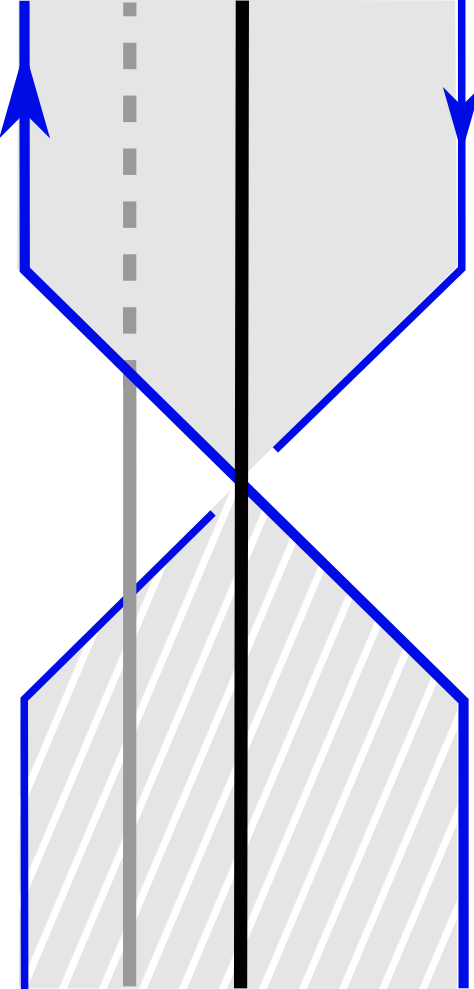}
	\hspace*{3cm}
	\includegraphics[trim = 0 30 0 0, clip, width=1.2cm]{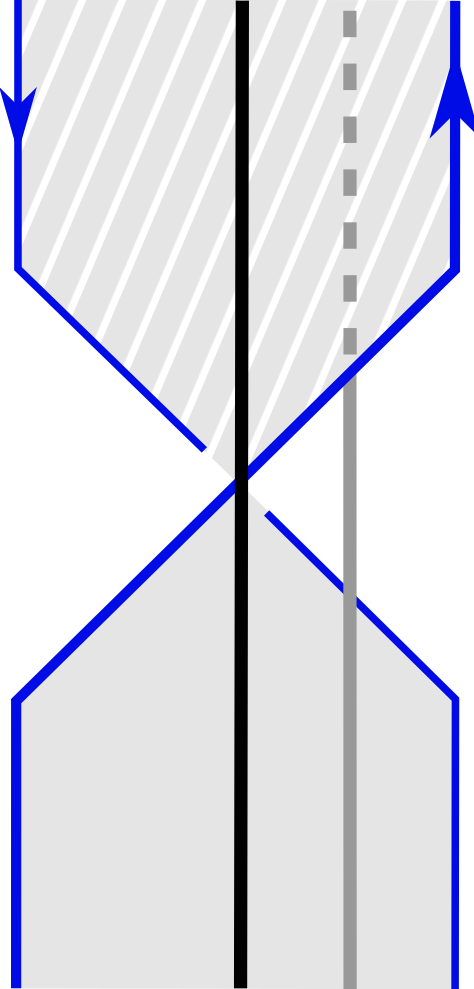}
	\hspace*{1cm}
	\includegraphics[trim = 0 30 0 0, clip, width=1.2cm]{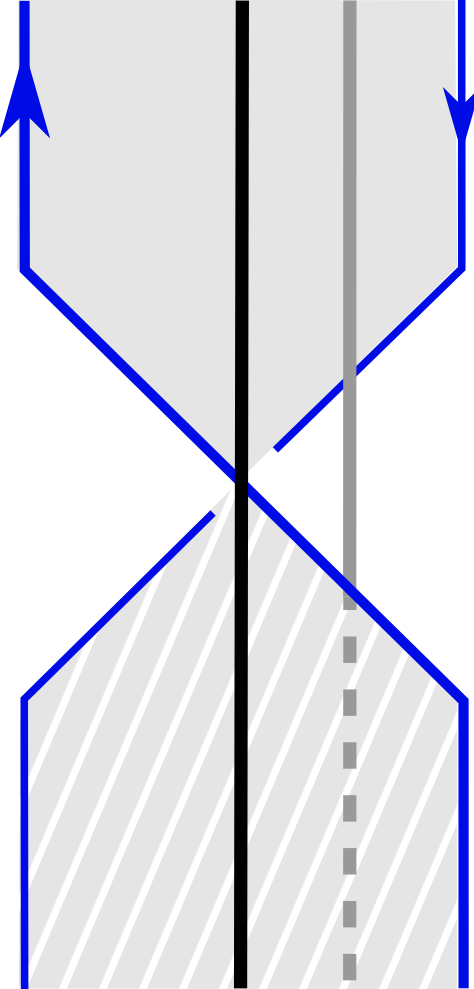}
	\caption{Convention for drawing the pushouts in the positive (left) and negative (right) directions near a crossing of~$L \#_1 rL$ occurring along an edge of~$D$: if the pushout appears behind the surface (dotted grey) then it is drawn between the original curve (black) and the diagram for~$L$ (thick blue); if the pushout appears in front of the surface (solid grey) then it is drawn between the curve and the diagram for~$rL$ (thin blue).}
	\label{fig:edgenolinking}
\end{figure}

 The local contribution to the matrices~$A^\varepsilon$ near different types of crossings of~$D$ are given in Figure~\ref{fig:locallinking}. Since the C-complex near a negative bichromatic crossing is the mirror image of the C-complex near a positive bichromatic crossing, the contribution near a negative bichromatic crossing can be obtained by changing the sign of the contribution near a positive crossing. Moreover, by the symmetry~\eqref{eq:A-sym}, the contribution for opposite choices of~$\varepsilon$ can be obtained by transposition. In conclusion, the local linking of all possible cases can be computed from those in Figure~\ref{fig:locallinking}. For the sake of clarity, let us mention that in the case of monochromatic crossings as well, one could compute the contribution for negative crossings from the one for positive crossings: the linking numbers involving the curves~$b$ and~$d$ (or their pushout) is unchanged, while the others change in a controlled way. However, since the precise relation is less immediately evident, we prefer to include both matrices in Figure~\ref{fig:locallinking}.

\begin{figure}[htbp]
    \centering
    
\begin{subfigure}{1 \textwidth}
    $$\begin{array}{c||cccc|c}
\lk & a & b & c & d & v\\
\hline
\hline
&&&&&\\[-0.3cm]
a^{(1)} & -1/2 & -1/2 & 0 & 0 & 1\\
b^{(1)} & 0 & 0 & 0 & 0 & 0\\
c^{(1)} & 0 & -1/2 & -1/2 & 0 & 1\\
d^{(1)} & 1/2 & 0 & 1/2 & 0 & -1\\
\hline
&&&&&\\[-0.3cm]
v^{(1)} &0 & 1 & 0 & 0 & -1
\end{array} 
\hspace*{1cm}\begin{array}{c||cccc|c}
\lk & a & b & c & d & v\\
\hline
\hline
&&&&&\\[-0.3cm]
a^{(1)} & 1/2 & -1/2 & 0 & 0 & 0\\
b^{(1)} & 0 & 0 & 0 & 0 & 0\\
c^{(1)} & 0 & -1/2 & 1/2 & 0 & 0\\
d^{(1)} & 1/2 & 0 & 1/2 & 0 & -1\\
\hline
&&&&&\\[-0.3cm]
v^{(1)} &-1 & 1 & -1 & 0 & 1
\end{array}$$
\caption{Contributions near a positive (left) and negative (right) monochromatic crossing of color~$j$,
where~$x^{(1)}$ stands for~$x^\varepsilon$ with~$\varepsilon_j =1$.}
\end{subfigure}

\begin{subfigure}{1 \textwidth}

$$\begin{array}{c||cccc}
\lk & a & b & c & d \\
\hline
\hline
&&&&\\[-0.3cm]
a^{(1,1)} & -1/2 & -1/2 & 0 & 0\\
b^{(1,1)} & 0 & 0 & 0 & 0 \\
c^{(1,1)} & 0 & 1/2 & -1/2 & 0 \\
d^{(1,1)} & 1/2 & -1 & 1/2 & 0 \\
\end{array}
\hspace*{1cm}
\begin{array}{c||cccc}
\lk & a & b & c & d \\
\hline
\hline
&&&&\\[-0.3cm]
a^{(1,-1)} & 0 & 0 & 0 & 0\\
b^{(1,-1)} & -1/2 & 1/2 & 0 & 0 \\
c^{(1,-1)} & 1 & -1/2 & 0 & -1/2 \\
d^{(1,-1)} & -1/2 & 0 & 0 & 1/2 \\
\end{array}$$
\caption{Contributions near a positive bichromatic crossing of colors~$(j,k)$,
with~$x^\varepsilon$ denoted by~$x^{(\varepsilon_j,\varepsilon_k)}$.}
\end{subfigure}
    \caption{The local contributions to linking numbers near crossings of~$D$, where the curves are labeled as in Figure~\ref{fig:H1classes}.}
    \label{fig:locallinking}
\end{figure}

We now write~$H(\om)$ as a sum over crossings of~$D$. For a crossing~$v$, let~$A_v^{\varepsilon}$ denote the square matrix (of size equal to the first Betti number of~$S$) given by the contribution to~$A^\varepsilon$ from the linking near~$v$; in other words,~$A^\varepsilon_v$ is zero everywhere except in the~$5 \times 5$ or~$4 \times 4$ minor corresponding to the homology classes coming into play near~$v$, where its values are given by the local contributions to linking numbers given by the matrices from Figure~\ref{fig:locallinking}. We have:
\begin{equation}
H(\omega) = \sum_{\varepsilon \in \{\pm 1\}^\mu} \bigg(\prod_{i=1}^\mu (1-\bom_i^{\varepsilon_i})\bigg)A^\varepsilon
= \sum_{v}\sum_{\varepsilon \in \{\pm 1\}^\mu} \bigg(\prod_{i=1}^\mu (1-\bom_i^{\varepsilon_i})\bigg) A_{v}^\varepsilon=:\sum_v H_v\,,
\label{eq:HoverXings}
\end{equation}
where~$H_v=\sum_{\varepsilon} \Big(\prod_{i}(1-\bom_i^{\varepsilon_i})\Big) A_{v}^\varepsilon$, and the sums indexed by~$v$ always refer to the sum over all crossing of~$D$.
Note that~$A_v^\varepsilon$ is entirely specified by~$\varepsilon_{j_v}$ and~$\varepsilon_{k_v}$. Thus we use~$A_v^{(\alpha, \beta)}$ to denote~$A_v^\varepsilon$ for any~$\varepsilon$ with~$\varepsilon_{j_v} = \alpha$ and~$\varepsilon_{k_v}=\beta$. If~$v$ is monochromatic with~$\varepsilon_{j_v} = \varepsilon_{k_v}=\alpha$, we simply write~$A_v^{(\alpha)}$. 

When~$v$ is monochromatic of color~$j=j_v$, the matrix~$H_v$ can be rewritten as
\begin{equation*}
\begin{split}
H_v  &= \sum_{(\varepsilon_1, \cdots, \widehat{\varepsilon}_{j}, \cdots, \varepsilon_\mu) \in \{\pm 1\}^{\mu -1}} \bigg(\prod_{\substack{i=1 \\ i\neq j}}^{\mu} (1-\bom_i^{\varepsilon_i})\bigg) \left((1-\bom_{j})A^{(1)}_v + (1-\om_{j})A^{(-1)}_v\right)\\
&=\bigg(\prod_{\substack{i=1 \\ i\neq j}}^{\mu}(1-\bom_i)(1-\om_i)\bigg
)\left((1-\bom_{j})A^{(1)}_v + (1-\om_{j})A^{(-1)}_v\right)\,,
\end{split}
\end{equation*}
where the second equality uses the relation~$(1-\bom_i)+(1-\om_i) =  (1-\bom_i)(1-\om_i)$.
Hence, using the notation~$s_i := (1-\om_i)$, the matrix~$H_v$ for a monochromatic crossing~$v$ of color~$j$ is given by
\begin{equation}
H_v =\bigg(\prod_{\substack{i=1 \\ i\neq j}}^{\mu} |s_i|^2 \bigg) \left(\bar{s}_{j}A^{(1)}_v + s_{j}A^{(-1)}_v\right)=:\bigg(\prod_{\substack{i=1 \\ i\neq j}}^{\mu} |s_i|^2\bigg) \,A_v\,,
\label{eq:Hvmono}
\end{equation}
while for a bichromatic crossing of colors~$(j,k)$, it is given by
\begin{equation}
H_v = \bigg(\prod_{\substack{i=1 \\ i\neq j,k}}^{\mu}  |s_i|^2\bigg)\left(\bar{s}_{j}\bar{s_{k}}A^{(1,1)}_v + \bar{s}_{j}s_{k}A^{(1,-1)}_v + s_{j}\bar{s}_{k}A^{(-1,1)}_v + s_{j}s_{k}A^{(-1,-1)}_v\right)=: \bigg(\prod_{\substack{i=1 \\ i\neq j,k}}^{\mu}  |s_i|^2\bigg) \,A_v\,,
\label{eq:Hvbi}
\end{equation}
where we use the notation
\[
A_v:=\begin{cases}
\bar{s}_{j}A^{(1)}_v + s_{j}A^{(-1)}_v&\text{if~$v$ is monochromatic of color~$j$};\\
\bar{s}_{j}\bar{s_{k}}A^{(1,1)}_v + \bar{s}_{j}s_{k}A^{(1,-1)}_v + s_{j}\bar{s}_{k}A^{(-1,1)}_v + s_{j}s_{k}A^{(-1,-1)}_v
&\text{if~$v$ is bichromatic of colors~$j,k$}\,.
\end{cases}
\]
Plugging Equations~\eqref{eq:Hvmono} and~\eqref{eq:Hvbi} into Equation~\eqref{eq:HoverXings}, we get
\begin{equation}\label{eq:HasA}
    H(\om) =  \sum_{v} H_v= \sum_{v} \bigg(\prod_{i \neq j_v, k_v}^{\mu} |s_i|^2\bigg) A_v\,.
\end{equation} 

Writing the Hermitian matrix~$H(\omega)$ as a block matrix of the form
$$
H(\omega) = \begin{array}{c||c|c}
     &\text{regions}& \text{mono crossings}  \\
     \hline
     \hline
     \text{regions}& X & Y\\
     \hline
     \text{mono crossings} & Y^* & Z\\
\end{array}
$$ 
we see that~$Z$ is a diagonal matrix with coefficient corresponding to the monochromatic crossing~$v$ given by~$-\Big(\prod_{i=1}^{\mu}|s_i|^2\Big)\sgn(v)$. In particular,~$Z$ is invertible (and hence has nullity~$\nul(Z)=0$), while its signature is equal to~$\sigma(Z)=-w_{\mathrm m}(D)$.
Furthermore,~$H(\omega)$ is congruent to the block diagonal matrix
\begin{equation}
\label{eq:MHM}
MH(\omega)M^*  = \begin{pmatrix} X-YZ^{-1}Y^* & 0 \\ 0 & Z^{-1}\end{pmatrix}
\end{equation}
via $M = \begin{pmatrix} I & -YZ^{-1} \\ 0 & Z^{-1}\end{pmatrix}$. Since $\sigma(Z^{-1}) = \sigma(Z) =-w_{\mathrm m}(D)$ and~$\nul(Z^{-1})=0$, it remains to show that the matrix~$X-YZ^{-1}Y^*$ coincides with~$\widetilde{\tau}_D(\om)$ up to transformations that do not affect the signature and nullity.

To determine~$X-YZ^{-1}Y^*$, let us fix two regions~$a$ and~$b$. Note that
$$(YZ^{-1}Y^*)_{a,b} = \sum_{v} Y_{a,v}Z_{v,v}^{-1}\overline{Y_{b,v}}\,,$$
and that it is only possible for both~$Y_{a,v}$ and~$Y_{b,v}$ to be nonzero if the regions~$a$ and~$b$ are both adjacent to the crossing~$v$. Therefore, the matrix~$YZ^{-1}Y^*$ is a sum over crossings, where the contribution at each crossing is a matrix that is zero everywhere except in the~$4 \times 4$ minor corresponding to the four adjacent regions of the crossing. The same is then true for~$X-YZ^{-1}Y^*$.
For a bichromatic crossing~$v$, there is no column of~$Y$ that corresponds to~$v$, so this~$4 \times 4$ minor is nothing but the nonzero~$4 \times 4$ minor of~$A_v$. For a monochromatic crossing~$v$, the~$4 \times 4$ minor of~$X-YZ^{-1}Y^*$ is obtained by performing the corresponding matrix operations to the~$5 \times 5$ minor of~$A_v$. The computation is similar to the single variable case, see~\cite{Liu23} for more details.
The explicit values for these local contributions to~$X-YZ^{-1}Y^*$ are given in Figure~\ref{fig:Av}.
Since the minor for a negative crossing turns out to be the negative of the minor for a positive one, we only provide these minors in the case of positive crossings.

\begin{figure}\centering
\begin{subfigure}{\textwidth}$$\left[\begin{array}{c||cccc}& a & b & c & d \\
\hline
\hline
&&&&\\[-0.3cm]
a& \frac{\om_j+\bom_j}{2}& -\frac{1+\bom_j}{2} & 1 & -\frac{1+\om_j}{2} \\
&&&&\\[-0.3cm]
b&-\frac{1+\om_j}{2} & 1 & -\frac{1+\om_j}{2} & \om_j \\
&&&&\\[-0.3cm]
c& 1 &-\frac{1+\bom_j}{2} &\frac{\om_j+\bom_j}{2} & -\frac{1+\om_j}{2}\\
&&&&\\[-0.3cm]
d&-\frac{1+\bom_j}{2}  & \bom_j& -\frac{1+\bom_j}{2}  & 1\end{array} \right]$$\caption{$v$ positive monochromatic crossing of color~$j$}\label{subfig:monopos}\end{subfigure}
\hfill
\begin{subfigure}{\textwidth}
$$\left[\begin{array}{c||cccc}& a & b & c & d \\
\hline
\hline
&&&&\\[-0.3cm]
a& -\frac{\ov{s_j}\ov{s_k}+s_js_k}{2} &\frac{\ov{s_k}(\om_j-\bom_j)}{2}& s_j\ov{s_k} &\frac{s_j(\bom_k-\om_k)}{2}\\
&&&&\\[-0.3cm]
b & \frac{s_k(\bom_j-\om_j)}{2} & \frac{s_j\ov{s_k}+\ov{s_j}s_k}{2} & \frac{s_j(\bom_k-\om_k)}{2} & -s_js_k \\
&&&&\\[-0.3cm]
c & \ov{s_j}s_k  & \frac{\ov{s_j}(\om_k-\bom_k)}{2} & -\frac{\ov{s_j}\ov{s_k}+s_js_k}{2} & \frac{s_k(\bom_j-\om_j)}{2}\\
&&&&\\[-0.3cm]
d & \frac{\ov{s_j}(\om_k-\bom_k)}{2} & -\ov{s_j}\ov{s_k}&\frac{\ov{s_k}(\om_j-\bom_j)}{2} & \frac{s_j\ov{s_k}+\ov{s_j}s_k}{2} \end{array} \right]$$\caption{$v$ positive bichromatic crossing of colors~$(j,k)$}\label{subfig:bipos}\end{subfigure}
\caption{The local contribution to~$X-YZ^{-1}Y^*$ for a positive monochromatic and bichromatic crossing~$v$. If~$v$ is a negative crossing, the matrix is the negative of the corresponding matrix for a positive crossing.}
\label{fig:Av}
\end{figure}

We need to perform one last change of basis, which we now describe. If the sublink~$L_i$ winds around the region~$a$ a total of~$\alpha_i$ times, then multiply the basis element corresponding to~$a$ with $\prod_{i=1}^\mu(-\bsom_i)^{\alpha_i}$.
This change of basis alters the matrices in Figure~\ref{fig:Av} in the following way. If~$v$ is a monochromatic
crossing of color~$j$, then~$L_j$ winds around~$b$ once more than it does around~$a$ and~$c$, and it winds around~$d$ once fewer. Thus the rows corresponding to~$b$ and~$d$ are multiplied by~$-\bsom_j$ and~$-\som_j$ respectively, and the columns are multiplied by the conjugates~$-\som_j$ and~$-\bsom_j$. If~$v$ is bichromatic crossing of colors~$(j, k)$, then~$L_j$ winds once more around~$a$ and~$b$ than it does around~$c$ and~$d$, and~$L_k$ winds once more around~$b$ and~$c$ than it does around~$a$ and~$d$. Thus we multiply the row for~$a$ by~$-\bsom_j$, the row for~$c$ by~$-\bsom_k$, and the row for~$b$ by~$\bsom_j\bsom_k$, and we multiply the corresponding columns by the conjugates. 

Remarkably, the sum of local contributions to~$X-YZ^{-1}Y^*$ from Figure~\ref{fig:Av} can now be written in terms of the single matrix~$\tau_v(\omega)$, the evaluation of~$\tau_v(x)$ at~$x_j=\Re(\omega_j^{1/2})$ and~$x_{jk}=\Re(\omega_j^{1/2}\omega_k^{1/2})$, in both the monochromatic and bichromatic cases. Indeed, for a monochromatic crossing~$v$, it coincides with~$\sgn(v)\tau_v(\omega)$,
while for a bichromatic crossing, it yields $$4\sgn(v)\sqrt{1-x_{j_v}^2}\sqrt{1-x_{k_v}^2}\tau_v(\omega)\,.$$
The result of the matrix~$X-YZ^{-1}Y^*$ after this change of basis thus gives
\begin{equation}
	\label{eq:finalMatrix}
\begin{split}
 \sum_{v \text{ mono}}  \bigg(\prod_{i \neq j_v} |s_i|^2\bigg) \sgn(v)\tau_v(\omega) + &
\sum_{v \text{ bi}} \bigg(\prod_{i \neq j_v, k_v}|s_i|^2\bigg) 4\sgn(v)\sqrt{1-x_{j_v}^2}\sqrt{1-x_{k_v}^2}\tau_v(\omega)  \\
= & \frac{1}{4}\bigg(\prod_{i=1}^\mu |s_i|^2\bigg)\;\sum_{v} \frac{\sgn(v)}{\sqrt{1-x_{j_v}^2}\sqrt{1-x_{k_v}^2}}\tau_v(\omega)=\frac{1}{4}\bigg(\prod_{i=1}^\mu |s_i|^2\bigg)\;\widetilde{\tau}_D(\om)\,.
\end{split}
\end{equation}
The positive constant~$\frac{1}{4}\prod_{i=1}^\mu |s_i|^2$ affects neither the signature nor the nullity, so we have
\begin{align*}
2\sigma_L(\om) &= \sigma(X-YZ^{-1}Y^*) + \sigma(Z^{-1}) =\sigma(\widetilde{\tau}_D(\om)) - w_{\mathrm{m}}(D)\,,\\
2\eta_L(\om) &= \eta(X-YZ^{-1}Y^*) + \eta(Z^{-1}) = \eta(\widetilde{\tau}_D(\om))\,,
\end{align*}
and the proof is complete.
\end{proof}

\begin{note}
	\label{note:change_variables}
In the manipulations of matrices throughout this proof, we never used the fact that the~$\om_j$ are complex numbers (except, of course, when computing signatures): the only property needed is that $\om_j\bom_j = 1$. Therefore, everything works equally well if we consider the $\om_j^{1/2}$ as \emph{formal variables}, and set $\bom_j^{1/2} := \om_j^{-1/2}$, $x_j=\Re(\om_j^{1/2}) := \frac{\om_j^{1/2}+\om_j^{-1/2}}{2}$, and $\sqrt{1-x_j^2}~=~\Im(\om_j^{1/2})~:=~\frac{\om_j^{1/2}-\om_j^{-1/2}}{2i}$.
\end{note}

\subsection{The Conway function}
\label{sub:proof-Conway}
	
In this section we discuss how to compute the Conway function of a colored link from the matrix~$\widetilde{\tau}_D(x)$ and prove the second point of Theorem~\ref{thm}, which we now restate for convenience.

\begin{prop}\label{prop:conway}
    If~$D$ is a connected~$\mu$-colored diagram for a~$\mu$-colored link~$L$, we have the equality
\[
\nabla^2_L(t_1,\dots,t_\mu)=\frac{1}{(t_1-t_1^{-1})^2}\bigg(\prod_v(-\sgn(v))\frac{t_j-t_j^{-1}}{2}\frac{t_k-t_k^{-1}}{2}\bigg)\cdot\det(\widetilde{\tau}_D(t^2))\,,
\]
where the product is over all crossings of~$D$, the indices~$j,k$ are the (possibly identical) colors of the two strands crossing at~$v$, and~$\tau_D(t^2)$ stands for the evaluation of~$\tau_D(x)$ at
\[
x_j=\frac{t_j+t_j^{-1}}{2}\,,\quad x_{jk}=\frac{t_jt_k+t_j^{-1}t_k^{-1}}{2}\,.
\]
\end{prop}

%



Our starting point is Equation~\eqref{eq:Conway}, which expresses the Conway function~$\nabla_L$
in terms of the matrix~$A_S$ associated to a connected C-complex $S$ for~$L$, recall~\eqref{eq:A}.
First note that if~$H(\omega)$ denotes the matrix defined in~\eqref{eq:H}, which is used for computing the multivariable signature, and we consider $\om_i^{-1/2} =: t_i$ as a \emph{formal variable}, we have


\begin{equation}\label{eq:H-Conway}
    \begin{split}
        H(t^{-2})&:=H(t_1^{-2},\cdots,t_{\mu}^{-2}) =  \sum_{\varepsilon\in\{\pm 1\}^\mu}\bigg(\prod_{i=1}^\mu(1-t_i^{2\varepsilon_i})\bigg)A^\varepsilon = \sum_{\varepsilon\in\{\pm 1\}^\mu}\bigg(\prod_{i=1}^\mu\varepsilon_it_i^{\varepsilon_i}\varepsilon_i(t_i^{-\varepsilon_i}-t_i^{\varepsilon_i})\bigg)A^\varepsilon \\
        & = \bigg(\prod_{i=1}^\mu -(t_i - t_i^{-1})\bigg)\sum_{\varepsilon\in\{\pm 1\}^\mu}\bigg(\prod_{i=1}^\mu\varepsilon_it_i^{\varepsilon_i}\bigg) A^\varepsilon  = (-1)^{\mu} \bigg(\prod_{i=1}^\mu (t_i - t_i^{-1})\bigg) A_S \,.
    \end{split}
\end{equation}
Hence, the Conway function can in fact be computed from the matrix~$H$.

Now, in order to prove Proposition~\ref{prop:conway}, we adopt the same strategy as in the computation of the signature: starting from a \emph{connected} diagram~$D$, we use the C-complex~$S$ for~$L\#_1rL$ constructed in the previous section to compute the Conway function of~$L\#_1rL$, and conclude by applying well-known formulas relating the Conway function of a connected sum to the Conway functions of the summands. Note that, by construction, requiring~$D$ to be connected precisely means that~$S$ is connected.

So, let~$D = D_1\cup\dots\cup D_{\mu}$ be a connected,~$\mu$-colored diagram of a~$\mu$-colored link $L$ and let $S = S_1 \cup \dots\cup S_{\mu}$ be the C-complex for~$L\#_1rL$ defined in the proof of Proposition~\ref{prop:sig}. As before, the notation~$\#_1$ stands for the connected sum performed along a component of color~$1$. Let~$n_{\mathrm m}$ and~$n_{\mathrm b}$ denote the number of monochromatic and bichromatic crossings of~$D$, respectively, and~$n = n_{\mathrm m}+n_{\mathrm b}$ the total number of crossings. Similarly, let~$n_{{\mathrm m},i}$ (resp.~$n_{{\mathrm b},i}$) denote the number of monochromatic (resp. bichromatic) crossings of~$D$ \emph{without} the color~$i$. Finally, recall that $\operatorname{sgn}(S)$ denotes the product of the signs of the clasps of~$S$ (see Figure~\ref{fig:clasp}).

\begin{lem}\label{lem:C-complex} With the notations above, the C-complex $S$ satisfies:
    \begin{enumerate}
        \item $\sgn{S} = \prod\limits_{v \text{ bichr.}}\sgn{v}$, where the product is taken over all bichromatic crossings.

        \item Its first Betti number is equal to~$b_1(S) = 
        n+n_{\mathrm m}$ 
        and is even.

        \item $\chi(S\setminus S_1) = - n_{{\mathrm b},1}-2n_{{\mathrm m},1}$ and~$\chi(S\setminus S_i) = 1 - n_{{\mathrm b},i}-2n_{{\mathrm m},i}$ for all~$i\neq 1$.
    \end{enumerate}
\end{lem}

\begin{proof} The first equality is clear by construction, since~$S$ has one clasp for each bichromatic crossing of~$D$, and the sign of the clasp is equal to the sign of the corresponding crossing.
To check the second point  let~$r$ denote the number of regions of~$D$. By construction, we have~$b_1(S)=(r-2)+n_{\mathrm m}$, while an Euler characteristic computations yields the equality~$r-2=n$.
Since~$n_{\mathrm b}$ is always even, it follows that~$b_1(S)= n+n_{\mathrm{m}}=n_{\mathrm b}+2n_{\mathrm m}$ is also even. As for the third point, one just needs to notice that~$S\setminus S_i$ deformation retracts onto a graph~$\Gamma_i$ constructed from (the planar projection of) the diagram~$D\setminus D_i$ by adding one loop to each monochromatic crossing and, if~$i\neq 1$, by removing one edge of color~$1$ (which corresponds to performing the connected sum). The number of vertices of~$\Gamma_i$ minus the number of its edges yields the result.
\end{proof}

To shorten our formulas, we denote $s_{\mathrm b} := \prod\limits_{v \text{ bichr.}}\sgn{v}$ and $s_{\mathrm m} := \prod\limits_{v \text{ mono.}}\sgn{v}$, where the product is taken over all bichromatic (resp. monochromatic) crossings of $D$. We are now ready to prove Proposition~\ref{prop:conway}.

\begin{proof}[Proof of Proposition~\ref{prop:conway}]

In what follows, we always evaluate~$H$ at~$\omega_i = t_i^{-2}$, considered as formal variables, and rely on Note~\ref{note:change_variables} to use the computations from the proof of Proposition~\ref{prop:sig} in this formal setting.
Recall that, by \eqref{eq:MHM} and the ensuing discussion, there exists a matrix $M'$ such that $$M'H(\omega)M'^* = \widetilde\tau'_D(\omega) \oplus Z^{-1}\,,$$ where the matrix
$$Z = \bigg(\prod_{i=1}^{\mu}(1-\omega_i)(1-\omega_i^{-1})\bigg)\operatorname{diag}(-\sgn(v))$$ is indexed by the monochromatic crossings of $D$, and $\det(M')=\det(M'^*)=\det(Z^{-1})$ --- in the notations from the proof of Proposition~\ref{prop:sig}, $\widetilde\tau'_D(\omega)$ is the matrix $X-YZ^{-1}Y^*$ after the final change of basis, as in \eqref{eq:finalMatrix}. Therefore, $$\det(H(\omega)) = \det(Z)\det(\widetilde\tau'_D(\omega))  = (-1)^{n_{\mathrm m}} s_{\mathrm m} \bigg(\prod_{i=1}^{\mu}(1-\omega_i)(1-\omega_i^{-1})\bigg)^{n_{\mathrm m}} \det(\widetilde\tau'_D(\omega))\,.$$
Evaluating at~$\omega_i = t_i^{-2}$, and using the equality~$(1-t_i^2)(1-t_i^{-2}) = -(t_i - t_i^{-1})^2$, we obtain
\begin{equation}\label{eq:H-tau}
    \det(H(t^{-2})) = (-1)^{n_{\mathrm m} + \mu n_{\mathrm m}} s_{\mathrm m} \bigg(\prod_{i=1}^{\mu} (t_i - t_i^{-1})^{2n_{\mathrm m}}\bigg) \det(\widetilde\tau'_D(t^2))\,. \end{equation}
Furthermore, since~$\widetilde\tau'_D$ is a matrix of size~$n$ and by~\eqref{eq:finalMatrix}, $$\widetilde\tau'_D(\omega) = \frac{1}{4}\bigg(\prod_{i=1}^{\mu}(1-\omega_i)(1-\omega_i^{-1})\bigg)\widetilde\tau_D(\omega)\,,$$
we have
\begin{equation}\label{eq:tau}
    \det(\widetilde\tau'_D(t^2)) = (-1)^{\mu n}\bigg(\prod_{i=1}^{\mu} (t_i - t_i^{-1})^{2n}\bigg)\det({\textstyle\frac{1}{4}}\widetilde\tau_D(t^2))\,.
\end{equation}
Putting everything together, and writing~$t$ for~$(t_1,\dots,t_\mu)$, we obtain
    \begin{equation*}
        \begin{split}
            \nabla^2_L(t)  =& \nabla_L(t)\nabla_{rL}(t) = \frac{1}{t_1-t_1^{-1}}\nabla_{L\#_1rL}(t) \\
             \overset{\eqref{eq:Conway}}{=} & \frac{1}{t_1-t_1^{-1}} (-1)^{b_1(S)}\sgn(S)\bigg(\prod_{i=1}^\mu (t_i-t_i^{-1})^{\chi(S\setminus S_i)-1}\bigg)\det(A_S) \\
             \overset{\eqref{eq:H-Conway}}{=} & \frac{s_{\mathrm b}}{t_1-t_1^{-1}}\bigg(\prod_{i=1}^\mu (t_i-t_i^{-1})^{\chi(S\setminus S_i)-1}\bigg)(-1)^{\mu b_1(S)}\bigg(\prod_{i=1}^\mu (t_i-t_i^{-1})^{-b_1(S)}\bigg)\det(H(t^{-2})) \\
             \overset{\eqref{eq:H-tau}}{=} & \frac{s_{\mathrm b}}{t_1-t_1^{-1}}\bigg(\prod_{i=1}^\mu (t_i-t_i^{-1})^{\chi(S\setminus S_i)-1-b_1(S)}\bigg)(-1)^{n_{\mathrm m} + \mu n_{\mathrm m}} s_{\mathrm m} \bigg(\prod_{i=1}^{\mu} (t_i - t_i^{-1})^{2n_{\mathrm m}}\bigg) \det(\widetilde\tau'_D(t^2)) \\
             \overset{\eqref{eq:tau}}{=} & \frac{(-1)^{n_{\mathrm m} + \mu n_{\mathrm m}}s_{\mathrm b}s_{\mathrm m}}{t_1-t_1^{-1}}\bigg(\prod_{i=1}^\mu (t_i-t_i^{-1})^{\chi(S\setminus S_i)-1-b_1(S)+2n_{\mathrm m}}\bigg)(-1)^{\mu n}\bigg(\prod_{i=1}^{\mu} (t_i - t_i^{-1})^{2n}\bigg)\det({\textstyle\frac{1}{4}}\widetilde\tau_D(t^2)) \\
             \overset{}{=} & \frac{(-1)^{n_{\mathrm m}}}{(t_1-t_1^{-1})^2}\frac{s_{\mathrm b}s_{\mathrm m}}{4^n}\bigg(\prod_{i=1}^\mu (t_i-t_i^{-1})^{-n_{{\mathrm b},i}-2n_{{\mathrm m},i}+n+n_{\mathrm m}}\bigg)\det(\widetilde\tau_D(t^2))\,,
        \end{split}
    \end{equation*}
    where in the first line we used Corollary~$2$ and Proposition~$5$ of~\cite{Cim04}. The following equalities derive, as indicated, from Equations~\eqref{eq:Conway},~\eqref{eq:H-Conway},~\eqref{eq:H-tau} and~\eqref{eq:tau}, together with the first point of Lemma~\ref{lem:C-complex} in the third line (with~\eqref{eq:H-Conway}), the second point in the fourth line (with~\eqref{eq:H-tau}), and the second and third points in the last line.

    To conclude, we note that the exponent $-n_{{\mathrm b},i}-2n_{{\mathrm m},i}+n+n_{\mathrm m}$ appearing in the last line is equal to the number of bichromatic crossings involving a strand of color $i$ plus twice the number of monochromatic crossings of color $i$. Therefore,
    $$ \frac{s_{\mathrm b}s_{\mathrm m}}{4^n}\prod_{i=1}^\mu (t_i-t_i^{-1})^{-n_{{\mathrm b},i}-2n_{{\mathrm m},i}+n+n_{\mathrm m}} = \prod_v\sgn(v)\frac{t_j-t_j^{-1}}{2}\frac{t_k-t_k^{-1}}{2}\,,$$
    where the product on the right-hand side is over all crossings of $D$ and the indices~$j,k$ are the two (possibly identical) colors of strands crossing at $v$. The proposition now follows from observing that~$(-1)^{n_{\mathrm{m}}}=(-1)^n$ since~$n=n_{\mathrm m}+n_{\mathrm b}$ and~$n_{\mathrm b}$ is even.
\end{proof}

	\subsection{The multivariate Kauffman model}
	\label{sub:Kauffman}

Having finished the proof of Theorem~\ref{thm}, we now turn our attention to Corollary~\ref{cor}.

Starting from a \emph{connected} diagram~$D$ of a~$\mu$-colored link, let~$K_D$ (or simply~$K$) be the matrix defined by the labels in Figure~\ref{fig:labels}, and~$\widetilde K$ the square matrix obtained from~$K$ by removing two columns corresponding to two adjacent regions of~$D$ separated by a strand of color~$1$.
Corollary~\ref{cor} is a direct consequence of Proposition~\ref{prop:conway} together with the following lemma.

\begin{lem}\label{lem:factorization}
    Let $S = (S_{v,v})$ be the diagonal matrix indexed by the crossings of~$D$ with coefficients $$S_{v,v} = \frac{-4\sgn(v)}{(t_j-t_j^{-1})(t_k-t_k^{-1})}\,,$$ where~$j$ and~$k$ are the colors of the two strands meeting at~$v$. Then, we have
    $$ \tau_D(t^2) = K^{\mathrm{T}} S K\,.$$
\end{lem}

We start by proving Corollary~\ref{cor}, before addressing the proof of Lemma~\ref{lem:factorization}.

\begin{proof}[Proof of Corollary~\ref{cor}]
    By Lemma~\ref{lem:factorization}, we have $\widetilde\tau_D(t^2) = \widetilde K^{\mathrm{T}} S \widetilde K$, yielding the equality  $$(\det \widetilde K)^2 = \det(S)^{-1}\det(\widetilde\tau_D(t^2))\,.$$ Since moreover $$\det(S)^{-1} = \prod_v\bigg(-\sgn(v)\frac{t_j-t_j^{-1}}{2}\frac{t_k-t_k^{-1}}{2}\bigg)\,,$$
    Proposition~\ref{prop:conway} implies $(\det \widetilde K)^2 = (t_1-t_1^{-1})^2 \nabla^2_L(t_1,\dots,t_\mu)$, and Corollary~\ref{cor} follows.   
\end{proof}


\begin{proof}[Proof of Lemma~\ref{lem:factorization}]
    Recall that~$K = (K_{v,f})$ is a matrix with rows indexed by the crossings of~$D$ and columns indexed by the regions of~$D$. Let us fix two regions~$f$ and~$g$ of~$D$. By definition, the corresponding coefficient of $K^{\mathrm{T}} S K$ is
    \begin{equation}
    \label{eq:K}
    (K^{\mathrm{T}} S K)_{f,g} = \sum_v \frac{-4\sgn(v)}{(t_j-t_j^{-1})(t_k-t_k^{-1})} K_{v,f}K_{v,g}\,,
    \end{equation}
    while the corresponding coefficient of~$\tau_D(t^2)$ is
        \begin{equation}
        \label{eq:tau'}
        (\tau_D(t^2))_{f,g} = \sum_v \frac{-4\sgn(v)}{(t_j-t_j^{-1})(t_k-t_k^{-1})} (\tau_v(t^2))_{f,g}\,.
            \end{equation}
    In both cases, the sum is over all crossings of~$D$ (and~$j,k$ denote the colors of the strands crossing at~$v$), but the only non-zero contributions come from the crossings adjacent to both~$f$ and~$g$.
    
    Comparing the labels of Figure~\ref{fig:minor} evaluated at~$x_j=\frac{t_j+t_j^{-1}}{2}$ and~$x_{jk}=\frac{t_jt_k+t_j^{-1}t_k^{-1}}{2}$ with the labels of Figure~\ref{fig:labels}, one notices an interesting relation. To state it precisely, let~$\Q(t)$ denote the field of fractions of $\Z[t_1^{\pm 1},\cdots,t_\mu^{\pm 1}]$, and let $\varphi:\Q(t)\rightarrow \Q(t)$ be the involution induced by~$t_i\mapsto t_i^{-1}$ for all~$i$. We claim that the following equality holds:
     \begin{equation}\label{eq:symmetry}
        (\tau_v(t^2))_{f,g} = \frac{K_{v,f}K_{v,g} + \varphi(K_{v,f}K_{v,g})}{2}\,.
    \end{equation}
The proof of this claim is divided into three cases, depending on the relative positions of the regions~$f$ and~$g$.
Let us first assume that~$f$ and~$g$ are two different regions of the same checkerboard color. In such a case, we
have~$K_{v,f}K_{v,g} = 1 = (\tau_v(t^2))_{f,g}$ for each crossing~$v$ incident to both~$f$ and~$g$, so~\eqref{eq:symmetry} holds.
Let us now assume that~$f$ and~$g$ are regions with different checkerboard colors, meeting at a crossing~$v$ with strands of colors~$j$ and~$k$. If~$f$ and~$g$ are adjacent to the strand of color~$j$ (resp.~$k$), we get $K_{v,f}K_{v,g} = t_k^{\pm 1}$ (resp.~$t_j^{\pm 1}$). Since the coefficient of~$\tau_v(t^2)$ is~$x_k = \frac{t_k+t_k^{-1}}{2}$ (resp.~$x_j=\frac{t_j+t_j^{-1}}{2}$), Equation~\eqref{eq:symmetry} holds in this case as well.
Finally, let us assume that~$f=g$. For a crossing~$v$ incident to~$f$, we get either $K_{v,f}^2 = (t_jt_k)^{\pm 1}$ or $K_{v,f}^2 = (t_j^{-1}t_k)^{\pm 1}$, depending on the position of~$f$ around $v$. Similarly, the corresponding coefficient of~$\tau_v(t^2)$ is either $x_{jk}=\frac{t_jt_k+t_j^{-1}t_k^{-1}}{2}$ or $2x_jx_k-x_{jk} = \frac{t_jt_k^{-1}+t_j^{-1}t_k}{2}$, respectively. This concludes the proof of~\eqref{eq:symmetry}.
    
Equations~\eqref{eq:K},~\eqref{eq:tau'} and~\eqref{eq:symmetry} immediately imply the equality
\[
\tau_D(t^2) = \frac{K^{\mathrm{T}} S K + \varphi(K^{\mathrm{T}} S K)}{2}\,,
\]
where~$\varphi$ is applied to matrices coefficientwise. To conclude the proof of Lemma~\ref{lem:factorization}, it remains to check that~$\varphi(K^{\mathrm{T}} S K) = K^{\mathrm{T}} S K$.
    This fact being suprisingly technical, we make it the object of a final separate lemma.
\end{proof}

\begin{lem}\label{lem:symmetry}
    Let~$\Q(t)$ denote the field of fractions of $\Z[t_1^{\pm 1},\cdots,t_\mu^{\pm 1}]$, and let $\varphi:\Q(t)\rightarrow \Q(t)$ be the involution induced by~$t_i\mapsto t_i^{-1}$ for all $i$. Then, we have the equality
    $\varphi(K^{\mathrm{T}} S K) = K^{\mathrm{T}} S K$.
\end{lem}

\begin{proof}
    We have to show that, for any two regions~$f$ and~$g$ of~$D$, the coefficient $$(K^{\mathrm{T}} S K)_{f,g} = \sum_v \frac{-4\sgn(v)}{(t_j-t_j^{-1})(t_k-t_k^{-1})} K_{v,f}K_{v,g}$$
    is invariant under $\varphi$, where the sum is taken over all crossings adjacent to both~$f$ and~$g$.
    We will consider several cases, according to
    the relative positions of~$f$ and~$g$ with respect to the crossings.
    
    First of all, if~$f$ and~$g$ are two regions of the same checkerboard color, each common crossing~$v$ contributes a term $\frac{-4\sgn(v)}{(t_j-t_j^{-1})(t_k-t_k^{-1})}$ to the coefficient~$(K^{\mathrm{T}} S K)_{a,b}$. Since all these terms are invariant under~$\varphi$, this case is checked.
    
    Next, suppose that~$f$ and~$g$ have different checkerboard colors. Each edge of~$D$ adjacent to both regions gives two contributions to the sum, one for each crossing adjacent to the edge. So, let us consider a common edge with endpoints~$v$ and~$v'$, and suppose without loss of generality that the colors and orientations of the strands are as in Figure~\ref{fig:symmetry}, left. The two contributions sum up to
    $$\frac{-4st_j^{-s}}{(t_i-t_i^{-1})(t_j-t_j^{-1})} + \frac{-4s't_k^{s'}}{(t_i-t_i^{-1})(t_k-t_k^{-1})}\,,$$ where $s = \sgn(v)$ and~$s'=\sgn(v')$. Proving that the term displayed above is invariant under~$\varphi$ is clearly equivalent to showing that
    $$G := st_j^{-s}(t_k-t_k^{-1}) + s't_k^{s'}(t_j-t_j^{-1})$$
    satisfies~$\varphi(G) = -G$.
    Expanding the products and denoting by~$\chi$ the characteristic function, one checks that~$G$ is equal to
    \begin{equation*}
       (t_jt_k - t_j^{-1}t_k^{-1}) (\chi_{s'=1}-\chi_{s=-1}) + (t_jt_k^{-1} - t_j^{-1}t_k) (\chi_{s=-1}-\chi_{s'=-1})\,,
    \end{equation*} 
    which is clearly antisymmetric, thus finishing this case.

    Finally, let us consider the diagonal coefficient corresponding to a region~$f$. Suppose that, when moving around the boundary of~$f$ counterclockwise, one encounters~$n$ crossings $v_1, \dots, v_n$ of respective
    signs~$s_1,\dots,s_n$. (It can happen that~$f$ abuts the same crossing from two sides, but since the
corresponding labels are added, our computations remain valid in this case.)
    Let us also number the edges of the boundary from~$1$ to~$n$ as in Figure~\ref{fig:symmetry}, right. To each edge, we assign a sign~$\varepsilon_i\in\{\pm 1\}$, where~$\varepsilon_i = 1$ if the edge~$i$ is oriented coherently with the counterclockwise orientation of the boundary of~$f$, and~$\varepsilon_i = -1$ otherwise. Without loss of generality, we can assume that all the edges have different colors, that we also denote by~$1,\dots,n$; in the general case, if two colors coincide, one simply needs to identify the corresponding variables in the following computations, a transformation which does not affect the symmetry.

 \begin{figure}
        \centering
\begin{picture}(220,100)
        \put(-1,32){$j$}
        \put(28,0){$k$}
        \put(70,0){$i$}
        \put(19,59){$\scriptstyle v$}
        \put(49,28){$\scriptstyle v'$}
        \put(15,20){$f$}
        \put(45,45){$g$}
        \put(0,0){\includegraphics[width = 8cm]{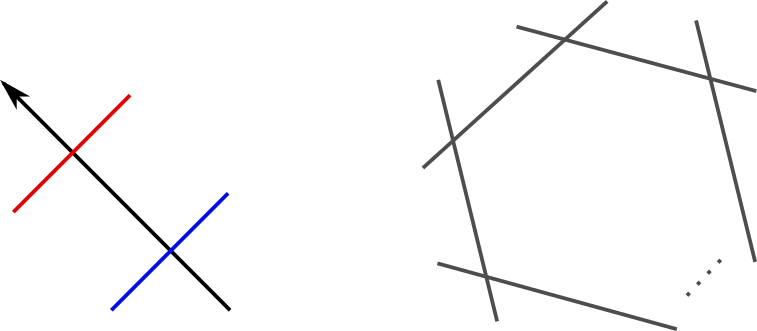}}
        \put(163,93){$\scriptstyle{v_1}$}
        \put(125,59){$\scriptstyle v_2$}
        \put(136,10){$\scriptstyle v_3$}
        \put(215,80){$\scriptstyle v_n$}
        \put(190,85){$1$}
        \put(146,73){$2$}
        \put(135,31){$3$}
        \put(165,0){$4$}
        \put(222,50){$n$}
        \put(180,40){$f$}
        \end{picture}
        \caption{The conventions in the proof of Lemma~\ref{lem:symmetry}.}
        \label{fig:symmetry}
    \end{figure}

    With these notations and the help of Figure~\ref{fig:labels}, one computes $$(K^{\mathrm{T}} S K)_{f,f} = \sum_{i=1}^n \frac{-4s_i(t_i^{\varepsilon_{i+1}}t_{i+1}^{-\varepsilon_i})^{s_i}}{(t_i-t_i^{-1})(t_{i+1}-t_{i+1}^{-1})}\,.$$
    As before, proving that this coefficient is invariant under~$\varphi$ is equivalent to showing that
    $$G := \prod_{j=1}^n (t_j-t_j^{-1}) \sum_{i=1}^n \frac{s_i(t_i^{\varepsilon_{i+1}}t_{i+1}^{-\varepsilon_i})^{s_i}}{(t_i-t_i^{-1})(t_{i+1}-t_{i+1}^{-1})}$$ satisfies~$\varphi(G) = (-1)^n G$. Expanding as a sum of monomials, we obtain
  \begin{align*}
         G  &= 
         \sum_{i=1}^n \bigg(\prod_{j\neq i,i+1} (t_j-t_j^{-1})\bigg)s_i(t_i^{\varepsilon_{i+1}}t_{i+1}^{-\varepsilon_i})^{s_i} = \sum_{i=1}^n \sum_{\alpha\in\{\pm 1\}^{n-2}} \bigg(\prod_{j\neq i,i+1} \alpha_j t_j^{\alpha_j}\bigg) s_i(t_i^{\varepsilon_{i+1}}t_{i+1}^{-\varepsilon_i})^{s_i}\\
             &= \sum_{\beta\in\{\pm 1\}^n} c_\beta t_1^{\beta_1}\cdots t_n^{\beta_n}
    \end{align*}
    for some coefficient~$c_\beta$. To compute these coefficients explicitly, let us define for each~$\beta\in\{\pm 1\}^n$ the (possibly empty) set of indices $I_\beta = \{i\in\{1,\dots,n\} \mid \beta_i =\varepsilon_{i+1}s_i, \beta_{i+1} = -\varepsilon_i s_i\}$. We then have
    \begin{equation*}
         c_\beta  = \sum_{i\in I_\beta} \beta_1\dots\beta_{i-1}s_i\beta_{i+2}\dots\beta_n = \sum_{i\in I_\beta} s_i\beta_i\beta_{i+1} (\beta_1\dots\beta_n) = \beta_1\dots\beta_n\sum_{i\in I_\beta} s_i\beta_i\beta_{i+1} =  \beta_1\dots\beta_n d_\beta\,,
    \end{equation*}
    with $d_\beta = \sum_{i\in I_\beta} s_i\beta_i\beta_{i+1}$.
The desired equality~$\varphi(G) = (-1)^n G$ is equivalent to~$c_{-\beta} = (-1)^n c_\beta$ for all~$\beta\in\{\pm 1\}^n$, which in turns is equivalent to~$d_{-\beta} = d_\beta$.

Therefore, we are left with the proof of the equality~$d_{-\beta} = d_\beta$ for all~$\beta\in\{\pm 1\}^n$.
Given any such~$\beta$ and any index~$i\in\{1,\dots,n\}$, define~$\tilde\beta$ by~$\tilde\beta_i = -\beta_i$ and~$\tilde\beta_j = \beta_j$ for~$j\neq i$. A straightforward but slightly cumbersome computation yields
    \begin{equation*}
         d_\beta - d_{\tilde\beta}=
        \begin{cases}
            -\varepsilon_i\beta_i, & \text{ if $(\beta_{i-1},\beta_{i+1}) = (-\varepsilon_i s_{i-1},-\varepsilon_i s_i)$}\\
            \varepsilon_i\beta_i, & \text{ if $(\beta_{i-1},\beta_{i+1}) = (\varepsilon_i s_{i-1},\varepsilon_i s_i)$}\\
            0, & \text{ otherwise. }
            \end{cases}
    \end{equation*}
    This expression is invariant if we replace~$\beta$ by~$-\beta$. It thus follows that, for any two~$\beta,\beta'\in\{\pm 1\}^n$, we have~$d_\beta - d_{\beta'} = d_{-\beta} - d_{-\beta'} $. To prove the equality~$d_{-\beta} = d_\beta$ for all~$\beta$, we therefore only need to check it for a single~$\beta$. Taking~$\beta = (1,\dots,1)$, we get
\[
d_\beta - d_{-\beta} = \sum_{i: (\varepsilon_i,\varepsilon_{i+1}) = (-s_i,s_i)} s_i - \sum_{i: (\varepsilon_i,\varepsilon_{i+1}) = (s_i,-s_i)} s_i = \sum_{i : \varepsilon_i\neq\varepsilon_{i+1}} \varepsilon_{i+1} = 0 \,,
\]
since there is an even number of crossings at which~$\varepsilon_i$ changes sign, going from~$1$ to~$-1$ in exactly half of the cases and from~$-1$ to~$1$ in the others.
\end{proof}

\subsection{The Alexander module}
\label{sub:module}

We conclude this article with a slightly informal discussion on yet another abelian link invariant,
namely the Alexander module (recall its definition from Section~\ref{sub:Conway}).

The question we address is whether it is possible to obtain a presentation of (the square of) the Alexander
module~$\mathcal{A}_L$ of a~$\mu$-colored link~$L$ from the matrix~$\widetilde\tau_D(x)$ associated
to a colored diagram~$D$ for~$L$. As we will see, the answer is positive in the case~$\mu=1$, but not in general.

\medskip

First, recall that~$\mathcal{A}_L$ does not admit a square presentation matrix over~$\Lambda=\Z[t_1^{\pm 1},\dots,t_\mu^{\pm 1}]$ if~$\Delta_L\neq 0$ and~$\mu\ge 4$~\cite{CroStr69}.
For this reason alone, there is no hope of answering the above question positively in general.
However, the module~$\mathcal{A}_L$ does admit a square presentation matrix over the localised ring
\[
\Lambda_S:=\Z[t_1^{\pm 1},\dots,t_\mu^{\pm 1},(t_1-1)^{-1},\dots,(t_\mu-1)^{-1}]\,.
\]
Moreover, by Corollary~3.6 of~\cite{CimFlo08}, the generalized Seifert matrices can be used to compute such a square presentation matrix. Therefore, it is natural to hope that the strategy developed in this work could be applied to this invariant as well. However, since the change of variables~$2x_j=t_j^{1/2}+t_j^{-1/2}$ makes use
of fractional powers of the variables, we will need to work over the slightly larger ring
\[
\Lambda'_S:=\Z[t_1^{\pm 1/2},\dots,t_\mu^{\pm 1/2},(t_1-1)^{-1},\dots,(t_\mu-1)^{-1}]\,.
\]

In the case~$\mu=1$, this program can be carried out, yielding the following result.
Let~$D$ be a connected diagram for an oriented link~$L$. As one easily checks, the coefficients of
the matrix~$\tau_D(x)$ are polynomials in~$2x=:y$. Let~$\mathcal{M}_D$ denote the~$\Z[y]$-module
presented by the matrix~$\widetilde\tau_D(y)$. Then, we have an isomorphism of~$\Lambda'_{S}$-modules
\[
\mathcal{M}_D\otimes_{\Z[y]}\Lambda'_{S}\simeq\mathcal{A}_L^{\oplus 2}\otimes_\Lambda\Lambda'_{S}\,,
\]
where~$\Lambda'_S=\Z[t^{\pm 1/2},(t-1)^{-1}]$ is a~$\Z[y]$-module via the ring homomorphism~$\Z[y]\to\Lambda'_S$ mapping~$y$ to~$t^{1/2}+t^{-1/2}$,
and a~$\Lambda$-module via the natural inclusion~$\Lambda\to\Lambda'_{S}$.
Less formally, one can say that the matrix~$\widetilde\tau_D(y)$ is a presentation matrix of~$\mathcal{A}_L^{\oplus 2}$ over~$\Lambda'_S$ via the substitution~$y=t^{1/2}+t^{-1/2}$.
Moreover, if~$L=K$ is a knot, then the multiplication by~$(t-1)$ is invertible in~$\mathcal{A}_K$.
As a consequence, the matrix~$\widetilde\tau_D(y)$ presents~$\mathcal{A}_K^{\oplus 2}$
over the ring~$\Z[t^{\pm 1/2}]$.

\medskip

Unfortunately, these results do not carry over to the case~$\mu>1$. Indeed,
let~$D$ be a~$\mu$-colored diagram for a~$\mu$-colored link~$L$, and assume that
each pair of colors meet in~$D$. Then, using Corollary~3.6 of~\cite{CimFlo08}, it is possible to prove that
the matrix~$\widetilde\tau_D(x)$ presents the Alexander module of~$L\#_1rL$ over
the ring
\[
\Z[{\textstyle{\frac{1}{2}}},t_1^{\pm 1/2},\dots,t_\mu^{\pm 1/2},(t_1-1)^{-1},\dots,(t_\mu-1)^{-1}]
\]
under the substitutions~$x_j=\frac{t_j+t_j^{-1}}{2}$ and~$x_{jk}=\frac{t_jt_k+t_j^{-1}t_k^{-1}}{2}$.
However, the isomorphism
\[
\mathcal{A}_{L\#_1 L'}\simeq \mathcal{A}_{L}\oplus \mathcal{A}_{L'}
\]
that we used in the case~$\mu=1$ is no longer valid in general for~$\mu>1$.
In other words, the addivity under connected sum enjoyed by the other abelian invariants considered in this work does not extend to the Alexander module in general.
  
\bibliographystyle{plain}
\bibliography{Bibliography}

\begin{thebibliography}{10}

\bibitem{Alex28}
James~W. Alexander.
\newblock Topological invariants of knots and links.
\newblock {\em Trans. Amer. Math. Soc.}, 30(2):275--306, 1928.

\bibitem{Cim04}
David Cimasoni.
\newblock A geometric construction of the {C}onway potential function.
\newblock {\em Comment. Math. Helv.}, 79(1):124--146, 2004.

\bibitem{CimFer23}
David Cimasoni and Livio Ferretti.
\newblock On the {K}ashaev signature conjecture.
\newblock {\em Fund. Math.}, in press.

\bibitem{CimFlo08}
David Cimasoni and Vincent Florens.
\newblock Generalized {S}eifert surfaces and signatures of colored links.
\newblock {\em Trans. Amer. Math. Soc.}, 360(3):1223--1264, 2008.

\bibitem{Con19}
Anthony Conway.
\newblock The {L}evine-{T}ristram signature: a survey.
\newblock In {\em 2019--20 {MATRIX} annals}, volume~4 of {\em MATRIX Book
  Ser.}, pages 31--56. Springer, Cham, [2021] \copyright 2021.

\bibitem{CNT}
Anthony Conway, Matthias Nagel, and Enrico Toffoli.
\newblock Multivariable signatures, genus bounds, and 0.5-solvable cobordisms.
\newblock {\em Michigan Math. J.}, 69(2):381--427, 2020.

\bibitem{Con70}
John~H. Conway.
\newblock An enumeration of knots and links, and some of their algebraic
  properties.
\newblock In {\em Computational {P}roblems in {A}bstract {A}lgebra ({P}roc.
  {C}onf., {O}xford, 1967)}, pages 329--358. Pergamon, Oxford-New York-Toronto,
  Ont., 1970.

\bibitem{Coo82}
Daryl Cooper.
\newblock The universal abelian cover of a link.
\newblock In {\em Low-dimensional topology ({B}angor, 1979)}, volume~48 of {\em
  London Math. Soc. Lecture Note Ser.}, pages 51--66. Cambridge Univ. Press,
  Cambridge-New York, 1982.

\bibitem{CroStr69}
R.~H. Crowell and D.~Strauss.
\newblock On the elementary ideals of link modules.
\newblock {\em Trans. Amer. Math. Soc.}, 142:93--109, 1969.

\bibitem{DMO}
Christopher~W. Davis, Taylor Martin, and Carolyn Otto.
\newblock Moves relating {C}-complexes: a correction to {C}imasoni's ``{A}
  geometric construction of the {C}onway potential function''.
\newblock {\em Topology Appl.}, 302:Paper No. 107799, 16, 2021.

\bibitem{Fri22}
Stefan Friedl, Chinmaya Kausik, and Jos\'{e}~Pedro Quintanilha.
\newblock An algorithm to calculate generalized {S}eifert matrices.
\newblock {\em J. Knot Theory Ramifications}, 31(11):Paper No. 2250068, 23,
  2022.

\bibitem{Har83}
Richard Hartley.
\newblock The {C}onway potential function for links.
\newblock {\em Commentarii mathematici Helvetici}, 58:365--378, 1983.

\bibitem{Kas21}
Rinat Kashaev.
\newblock On symmetric matrices associated with oriented link diagrams.
\newblock In {\em Topology and geometry—a collection of essays dedicated to
  Vladimir G. Turaev}, volume~33 of {\em {IRMA} Lect. Math. Theor. Phys.},
  pages 131--145. Eur. Math. Soc., Zürich, 2021.

\bibitem{Kau81}
Louis~H. Kauffman.
\newblock The {C}onway polynomial.
\newblock {\em Topology}, 20(1):101--108, 1981.

\bibitem{Kau83}
Louis~H. Kauffman.
\newblock {\em Formal knot theory}, volume~30 of {\em Mathematical Notes}.
\newblock Princeton University Press, Princeton, {NJ}, 1983.

\bibitem{Liu23}
Jessica Liu.
\newblock A proof of the {K}ashaev signature conjecture, 2023.

\bibitem{Mur70}
Kunio Murasugi.
\newblock On the signature of links.
\newblock {\em Topology}, 9:283--298, 1970.

\bibitem{Sat11}
Masashi Sato.
\newblock On the {C}onway potential function introduced by {K}auffman, 2011.

\bibitem{Zib17}
Claudius Zibrowius.
\newblock {\em On a Heegaard Floer theory for tangles}.
\newblock Phd thesis, University of Cambridge, 2017.

\bibitem{Zib19}
Claudius Zibrowius.
\newblock Kauffman states and {H}eegaard diagrams for tangles.
\newblock {\em Algebr. Geom. Topol.}, 19(5):2233--2282, 2019.

\end{thebibliography}
	
\end{document}